\pdfoutput=1
\RequirePackage{ifpdf}
\ifpdf 
\documentclass[pdftex]{sigma}
\else
\documentclass{sigma}
\fi

\numberwithin{equation}{section}

\newtheorem{Theorem}{Theorem}[section]
\newtheorem*{Theorem*}{Theorem}
\newtheorem{Corollary}[Theorem]{Corollary}
\newtheorem{Lemma}[Theorem]{Lemma}
\newtheorem{Proposition}[Theorem]{Proposition}
 { \theoremstyle{definition}

\newtheorem{Remark}[Theorem]{Remark} }

\begin{document}

\allowdisplaybreaks

\newcommand{\arXivNumber}{2208.05526}

\renewcommand{\PaperNumber}{041}

\FirstPageHeading

\ShortArticleName{Skew Symplectic and Orthogonal Schur Functions}

\ArticleName{Skew Symplectic and Orthogonal Schur Functions}

\Author{Naihuan JING~$^{\rm a}$, Zhijun LI~$^{\rm b}$ and Danxia WANG~$^{\rm b}$}

\AuthorNameForHeading{N.~Jing, Z.~Li and D.~Wang}

\Address{$^{\rm a)}$~Department of Mathematics, North Carolina State University, Raleigh, NC 27695, USA}
\EmailD{\href{mailto:jing@ncsu.edu}{jing@ncsu.edu}}

\Address{$^{\rm b)}$~School of Science, Huzhou University, Huzhou, Zhejiang 313000, P.R.~China}
\EmailD{\href{mailto:zhijun1010@163.com}{zhijun1010@163.com}, \href{mailto:dxwangmath@126.com}{dxwangmath@126.com}}

\ArticleDates{Received August 28, 2023, in final form May 12, 2024; Published online May~21, 2024}

\Abstract{Using the vertex operator representations for symplectic and orthogonal Schur functions, we define two families of symmetric functions and show that
they are the skew symplectic and skew orthogonal Schur polynomials defined implicitly by Koike and Terada and satisfy the general branching rules.
Furthermore, we derive the Jacobi--Trudi identities and Gelfand--Tsetlin patterns for these symmetric functions. Additionally, the vertex operator method yields their Cauchy-type identities. This demonstrates that vertex operator representations serve not only as a tool for studying symmetric functions but also offers unified realizations for skew Schur functions of types A, C, and D.}


\Classification{05E05; 17B37}

\section{Introduction}

For a partition $\lambda=(\lambda_1,\lambda_2,\ldots,\lambda_n)$, Schur functions, symplectic Schur functions, and orthogonal Schur functions in finitely many variables are respectively defined by~\cite{Mac1995,Me2019,Sta1999,Su1990}
\begin{align}
\label{e:s50}&{\rm s}_\lambda(x)=\frac{\det\big(x^{\lambda_j+n-j}_i\big)_{1\leq i,j\leq n}}{\det\big(x^{n-j}_i\big)_{1\leq i,j\leq n}},\\
\label{e:sp50}&{\rm sp}_\lambda\big(x^{\pm}\big)=\frac{\det\big(x^{\lambda_j+(n-j+1)}_i-x^{-\lambda_j-(n-j+1)}_i\big)^n_{i,j=1}}{\det\big(x^{n-j+1}_i-x^{-(n-j+1)}_i\big)^n_{i,j=1}},\\
\label{e:o50}&{\rm o}_\lambda\big(x^{\pm}\big)=2(-1)^{\frac{n(n-1)}{2}}\frac{\det\big(x^{-\lambda_j+j-n}_i+\delta_{\lambda_n\neq0}\delta_{j\neq n}x^{\lambda_j-j+n}_i\big)^n_{i,j=1}}{\det\big(x^{-j+1}_i+x^{j-1}_i\big)^n_{i,j=1}},
\end{align}
where $\delta_{a\neq b}$ equals 1 if $a\neq b$ and is $0$ otherwise.
They appear in various contexts of mathematics, especially in the representation theory of classical Lie groups (general linear groups, symplectic groups and
even orthogonal groups) as characters of finite-dimensional irreducible representations~\cite{KT1987,Lw1950,Su1990,Wey1946}. They are also
studied in generalized vertex algebras and $q$-characters~\cite{CG2020, LTW2016}.

With the advent of Schur symplectic and Schur orthogonal symmetric functions, arises the inquiry into skew versions for these symmetric functions.
Since the family of symplectic Schur or orthogonal Schur functions indexed by partitions does not constitute a complete basis in the ring of symmetric functions (in fact they belong to the ring of Laurent polynomials), the traditional approach of employing the adjoint operation is not viable for defining skew orthogonal or symplectic Schur functions. Instead the skew symplectic and (even) skew orthogonal symmetric functions were implicitly defined by
Koike and Terada~\cite{KT1990} by restricting the irreducible characters of the classical groups into their subgroups and were
expressed by tableaux representations. In~\cite{AF2020}, the tableaux representations were transformed into the Gelfand--Tsetlin pattern
representations by Ayyer and Fischer. It is natural to ask whether the skew symplectic Schur and orthogonal Schur functions can be computed by
Jacobi--Trudi type formulas, which would then lead to a parallel formulism as the Schur functions and skew Schur functions and also provide
a~lifting method to define the skew Schur symplectic and orthogonal functions in infinitely many variables.

The {\it aim of this paper} is to address these questions by
the vertex operator
realization of the classical symmetric functions~\cite{Jing1991} and the symplectic Schur and orthogonal Schur functions~\cite{Ba1996,JN2015}. The vertex algebraic method of studying Schur polynomials
can be traced back to the Kyoto school's work on integrable systems~\cite{DJKM} where the tau functions of the KP hierarchy are in terms of Schur polynomials.
Skew Schur functions ${\rm s}_{\lambda/\mu}$ satisfy the celebrated decomposition formula, i.e., the general branching rule
\begin{align*}
&{\rm s}_\lambda(x_1,\dots,x_n)=\sum_{\mu\subset\lambda}{\rm s}_\mu(x_1,\dots,x_{n-k}){\rm s}_{\lambda/\mu}(x_{n-k+1},\dots,x_{n}),
\end{align*}
which can be interpreted as the restriction of the irreducible representation of $\mathrm{GL}(n,\mathbb{C})$ to $\mathrm{GL}(n-k,\mathbb{C})\times \mathrm{GL}(k,\mathbb{C})$.
We will define skew (symplectic/orthogonal) Schur functions using vertex operators and derive their combinatorial properties such as the Jacobi--Trudi identities, Cauchy-type identities and verify they satisfy the Gelfand--Tsetlin pattern representations found by~\cite{AF2020}.
We also show that our vertex operator representations for skew-type symmetric functions comply with the general branching rules in the agreement with
Koike and Terada's approach.

The construction of skew Schur functions and their generalizations can be formulated in terms of the representation theory of infinite-dimensional Lie algebras.
One starts with the infinite-dimensional Heisenberg algebra~\cite{FK1980} with the center $1$ and defines the Fock space $\mathcal{M}^*$ and its completion $\widetilde{\mathcal{M}^*}$ generated by the vacuum $\langle0|$ (which is the vacuum vector in the right module) and the Heisenberg generators with positive modes. There are three families of special elements ($\langle\lambda|$, $\langle\lambda^{\rm sp}|$ and $\langle\lambda^{\rm o}|$
in the right module) defined by partitions $\lambda$, corresponding respectively to the Schur, symplectic Schur, and orthogonal Schur cases. We will prove that they are orthonormal by computing the inner product with the elements ($|\lambda\rangle$, $|\lambda^{\rm sp}\rangle~and~|\lambda^{\rm o}\rangle$ in the left module) in the Fock space $\mathcal{M}$ respectively (see Theorem~\ref{th2}), which is key to construct (skew) symmetric functions and also savages the problem of adjoint operation with skew versions. For the half vertex operators $\Gamma_+(\{x\})$ and $\Gamma_+\big(\big\{x^{\pm}\big\}\big)$, we will show that
\begin{align*}
&\langle 0|\Gamma_+(\{x\})=\sum_{\lambda}\langle \lambda|{\rm s}_\lambda(x),\\
&\langle 0|\Gamma_+\big(\big\{x^{\pm}\big\}\big)=\sum_{\lambda}\langle \lambda^{\rm sp}|{\rm sp}_\lambda\big(x^{\pm}\big)=\sum_{\lambda}\langle \lambda^{\rm o}|{\rm o}_\lambda\big(x^{\pm}\big),
\end{align*}
where each summation is over (generalized) partitions $\lambda=(\lambda_1,\dots,\lambda_N)$ with zeros parts allowed in the end. The benefit of this approach is that we can easily obtain the well-known classical Cauchy identities for (symplectic/orthogonal) Schur functions~\cite{KT1987,Su1990}. This approach is logically independent of the Jacobi--Trudi formula.

It is well known that $\langle \mu|\Gamma_+(\{x\})|\lambda\rangle={\rm s}_{\lambda/\mu}(x)$ \cite{Lam2006,Oko2001} using the half vertex operator
$\Gamma_+(\{x\})$. We extend the vertex operator method to derive two families of symmetric functions
\begin{align*}
&{\rm sp}_{\lambda/\mu}\big(x^{\pm}\big)=\langle\mu^{\rm sp}|\Gamma_+\big(\big\{x^{\pm}\big\}\big)|\lambda^{\rm sp}\rangle,\\
&{\rm o}_{\lambda/\mu}\big(x^{\pm}\big)=\langle\mu^{\rm o}|\Gamma_+\big(\big\{x^{\pm}\big\}\big)|\lambda^{\rm o}\rangle,
\end{align*}
which can be interpreted as the skew symplectic/orthogonal Schur functions from the general branching rule. We also show that each skew function can be written as a Jacobi--Trudi type determinant of complete homogeneous symmetric functions with variables $x_1,x^{-1}_1,\dots,x_N,x^{-1}_N$.

Exploiting the combinatorial properties of $\langle\mu^{\rm sp}|$ and $\langle\mu^{\rm o}|$, we also obtain Gelfand--Tsetlin pattern representations (generating
functions of some Young tableaux in an equivalent form~\cite{Sta1999}) for skew symplectic/orthogonal Schur functions, which provide alternative definitions of
symplectic/orthogonal Schur functions (see~\cite{AF2020}). The vertex operator method depending on the combinatorial properties of some orthogonal elements of the Fock space $\widetilde{\mathcal{M}^*}$ may shed light on specific descriptions of generalized shape $\pi$ symmetric functions~\cite{FJK2010}.

This paper is organized as follows. In Section~\ref{s2}, we recall the vertex operators related to fermionic Fock spaces $\mathcal{M}$ and $\mathcal{M}^*$
and prove the orthonormality for three distinguished sets of~$\mathcal{M}$ and the completion of $\mathcal{M}^*$. In Section~\ref{s3}, we define the skew symplectic/orthogonal Schur functions by the vertex algebraic method and derive their Jacobi--Trudi identities. We also show the vertex operator realizations for these skew symmetric functions are consistent with the general branching rule, therefore they agree with the skew symplectic and orthogonal Schur polynomials introduced by Koike and Terada. In Section~\ref{s4}, we give formulas of skew symplectic/orthogonal Schur functions in terms of Gelfand--Tsetlin pattern representations, which agree with Ayyer and Fischer's results. In Section~\ref{s5}, we provide a vertex operator approach to the Cauchy identities for the (symplectic/orthogonal) Schur functions, which is logically independent of the
Jacobi--Trudi formulas~\cite{JR2016}. We also obtain Cauchy-type identities for skew (symplectic/orthogonal) Schur functions.

\section{Preliminaries}\label{s2}

Let $\mathcal{H}$ be the Heisenberg algebra generated by $\{a_n\mid n\neq 0\}$ with the central element $c=1$ subject to the commutation relations \cite{FK1980}
\begin{align}\label{e:he1}
[a_m,a_n]=m\delta_{m,-n}c,\qquad [a_n,c]=0.
\end{align}

The Fock space $\mathcal{M}$ (resp.\ $\mathcal{M}^*$) is generated by the vacuum vector $|0\rangle$ (resp.\ dual vacuum vector $\langle0|$) and
subject to
\begin{align*}
a_n|0\rangle=\langle0|a_{-n},\qquad n>0.
\end{align*}
In other words, $\mathcal{M}$ and $\mathcal{M}^*$ are respectively left and right modules for $\mathcal H$.
It is easy to see that $\mathcal{H}$ acts on $\mathcal{M}$ (resp.\ $\mathcal{M}^*$) irreducibly. Moreover, $\mathcal{M}$ (resp.\ $\mathcal{M}^*$) is isomorphic to the symmetric algebra $\mathbb{C}[a_{-1},a_{-2},\dots]$ (resp.\ $\mathbb{C}[a_1,a_2,\dots]$). Note that
$\mathcal M^*$ is a graded space with the gradation induced from that of $\mathcal H$.
Let \smash{$\big\{\mathcal M^*_n\big\}$} be the filtration of the subspaces $\mathcal M^*_n$ spanned by homogeneous elements with degree $\geq n$. Then we
let $\widetilde{\mathcal{M}^*}$ be the associated completion of the Fock space $\mathcal{M}^*$.

Consider the following vertex operators given in~\cite{Ba1996,JN2015} (see also~\cite{SZ2006})
\begin{align}
&X(z)=\exp\Biggl(\sum^\infty_{n=1}\frac{a_{-n}}{n}z^n\Biggr)\exp\Biggl(-\sum^\infty_{n=1}\frac{a_n}{n}z^{-n}\Biggr)=\sum_{n\in \mathbb{Z}}X_nz^{-n},\nonumber\\
&X^*(z)=\exp\Biggl(-\sum^\infty_{n=1}\frac{a_{-n}}{n}z^n\Biggr)\exp\Biggl(\sum^\infty_{n=1}\frac{a_n}{n}z^{-n}\Biggr)=\sum_{n\in \mathbb{Z}}X^*_nz^n,\nonumber\\
&Y(z)=\exp\Biggl(\sum^\infty_{n=1}\frac{a_{-n}}{n}z^n\Biggr)\exp\Biggl(-\sum^\infty_{n=1}\frac{a_n}{n}(z^{-n}+z^n)\Biggr)=\sum_{n\in \mathbb{Z}}Y_nz^{-n},\nonumber\\
&Y^*(z)=\big(1-z^2\big)\exp\Biggl(-\sum^\infty_{n=1}\frac{a_{-n}}{n}z^n\Biggr)\exp\Biggl(\sum^\infty_{n=1}\frac{a_n}{n}(z^{-n}+z^n)\Biggr)=\sum_{n\in \mathbb{Z}}Y^*_nz^n,\nonumber\\
&W(z)=\big(1-z^2\big)\exp\Biggl(\sum^\infty_{n=1}\frac{a_{-n}}{n}z^n\Biggr)\exp\Biggl(-\sum^\infty_{n=1}\frac{a_n}{n}(z^{-n}+z^n)\Biggr)=\sum_{n\in \mathbb{Z}}W_nz^{-n},\nonumber\\
&W^*(z)=\exp\Biggl(-\sum^\infty_{n=1}\frac{a_{-n}}{n}z^n\Biggr)\exp\Biggl(\sum^\infty_{n=1}\frac{a_n}{n}(z^{-n}+z^n)\Biggr)=\sum_{n\in \mathbb{Z}}W^*_nz^n. \label{e:ve1}
\end{align}
\begin{Lemma}The actions of the operators $X_n$, $X^*_n$, $Y_n$, $Y^*_n$, $W_n$ and $W^*_n$ on the $($dual$)$ vacuum are given by
\begin{gather}
X_n|0\rangle=X^*_{-n}|0\rangle=\langle 0|X_{-n}=\langle 0|X^*_n \nonumber\\
\label{e:re1} \hphantom{X_n|0\rangle}{}
=Y_n|0\rangle=Y^*_{-n}|0\rangle=W_n|0\rangle=W^*_{-n}|0\rangle=0 \qquad \text{for}~n>0,\\
\label{e:re2}\langle 0|Y^*_n=-\langle 0|Y^*_{-n+2},\quad \langle 0|Y_n=\langle 0|Y_{-n},\quad \langle 0|W_n=-\langle 0|W_{-n-2},\quad \langle 0|W^*_n=\langle 0|W^*_{-n}.
\end{gather}
\end{Lemma}
\begin{proof}
Relations \eqref{e:re1} are obvious, and we only need to prove \eqref{e:re2}. Note the fact
\begin{gather*}
\langle 0|\exp\Biggl(\sum^\infty_{n=1}\frac{a_n}{n}(z^{-n}+z^n)\Biggr)=\langle 0|\frac{1}{1-z^2}\exp\Biggl(\sum^\infty_{n=1}\frac{a_{-n}}{n}z^n\Biggr)Y^*(z)
=\frac{1}{1-z^2}\langle 0|Y^*(z),\\
\langle 0|\exp\Biggl(\sum^\infty_{n=1}\frac{a_n}{n}(z^{-n}+z^n)\Biggr)=\langle 0|\frac{1}{1-z^{-2}}\exp\Biggl(\sum^\infty_{n=1}\frac{a_{-n}}{n}z^{-n}\Biggr)Y^*\big(z^{-1}\big)\\
\hphantom{\langle 0|\exp\Biggl(\sum^\infty_{n=1}\frac{a_n}{n}(z^{-n}+z^n)\Biggr)}{}
=\frac{-z^2}{1-z^2}\langle 0|Y^*\big(z^{-1}\big),
\end{gather*}
therefore
\begin{align}
\label{e:re3}\langle 0|Y^*(z)=-z^2\langle 0|Y^*\big(z^{-1}\big).
\end{align}
Taking the coefficient of $z^n$ of \eqref{e:re3}, we have
\begin{align*}
\langle 0|Y^*_n=-\langle 0|Y^*_{-n+2}.
\end{align*}
Similarly, we can get the other commutation relations in \eqref{e:re2}.
\end{proof}

The following result can be found in~\cite[Theorem~3.4, Proposition~3.6]{JN2015}.
\begin{Lemma}The components of the vertex operators \eqref{e:ve1} satisfy the following commutation relations
\begin{alignat}{4}
&X_iX_j+X_{j+1}X_{i-1}=0,\quad\ && X^*_iX^*_j+X^*_{j-1}X^*_{i+1}=0,\quad\ && X_iX^*_j+X^*_{j+1}X_{i+1}=\delta_{i,j},&\nonumber\\
&Y_iY_j+Y_{j+1}Y_{i-1}=0,\quad\ && Y^*_iY^*_j+Y^*_{j-1}Y^*_{i+1}=0,\quad\ && Y_iY^*_j+Y^*_{j+1}Y_{i+1}=\delta_{i,j},&\label{e:com}\\
&W_iW_j+W_{j+1}W_{i-1}=0,\quad\ && W^*_iW^*_j+W^*_{j-1}W^*_{i+1}=0,\quad\ && W_iW^*_j+W^*_{j+1}W_{i+1}=\delta_{i,j}.&\nonumber
\end{alignat}
\end{Lemma}
A ${\it partition}$ $\lambda=(\lambda_1,\lambda_2,\dots,\lambda_l)$ is a~weakly decreasing sequence of positive integers, and
a~{\it generalized partition} $\lambda=(\lambda_1,\lambda_2,\dots,\lambda_n)$ is a weakly decreasing finite sequence of nonnegative integers. The $\lambda_i$'s are
called the parts of $\lambda$, and the number $l(\lambda)$ of nonzero parts is the length of $\lambda$. If $\lambda$, $\mu$ are partitions, we will write $\mu\subset\lambda$ to mean that $\lambda_i\geq \mu_i$ for all $i\geq 1$. Therefore, a~generalized partition is a partition appended with
a string of finitely many zeros.\footnote{In the paper, generalized partitions should be used always for the skew symplectic and orthogonal Schur functions. We remark that the skew Schur functions
indexed by generalized partitions coincide with those indexed by ordinary partitions.} For a~given (generalized) partition $\lambda=(\lambda_1,\lambda_2,\dots,\lambda_l)$, let
\begin{alignat}{3}
\label{e:sc2}&|\lambda\rangle=X_{-\lambda_1}X_{-\lambda_2}\cdots X_{-\lambda_l}|0\rangle,\qquad&& \langle \lambda|=\langle 0|X^*_{-\lambda_l}\cdots X^*_{-\lambda_1},&\\
\label{e:sp6}&|\lambda^{\rm sp}\rangle=Y_{-\lambda_1}Y_{-\lambda_2}\cdots Y_{-\lambda_l}|0\rangle, \qquad&&\langle \lambda^{\rm sp}|=\langle 0|Y^*_{-\lambda_l}\cdots Y^*_{-\lambda_1},&\\
\label{e:o6}&|\lambda^{\rm o}\rangle=W_{-\lambda_1}W_{-\lambda_2}\cdots W_{-\lambda_l}|0\rangle,\qquad&&\langle \lambda^{\rm o}|=\langle 0|W^*_{-\lambda_l}\cdots W^*_{-\lambda_1}.&
\end{alignat}

We remark that the vectors $|\lambda\rangle$, $|\lambda^{\rm sp}\rangle$ and $|\lambda^{\rm o}\rangle$ are in $\mathcal{M}$, $\langle \lambda|$ in $\mathcal{M}^*$, while the vectors $\langle \lambda^{\rm sp}|$ and $\langle \lambda^{\rm o}|$ belong to the completion $\widetilde{\mathcal{M}^*}$. Due to the fact
\begin{align*}
\langle 0|Y^*_0\neq \langle 0|,\qquad \langle 0|W^*_0\neq \langle 0|,
\end{align*}
the vector $\langle \lambda^{\rm sp}|$ (resp.\ $\langle \lambda^{\rm o}|$) differs from the vector $\langle \mu^{\rm sp}|$ (resp.\ $\langle \mu^{\rm o}|$) even if $\lambda$ and $\mu$ only differ by a string of zeros at the end.
For examples,
\begin{align*}
&\langle 0|Y^*_{0}Y^*_{-1}Y^*_{-4}\neq\langle 0|Y^*_{-1}Y^*_{-4}\neq\langle 0|Y^*_{0}Y^*_{0}Y^*_{-1}Y^*_{-4},\\
&\langle 0|W^*_{0}Y^*_{-1}W^*_{-4}\neq\langle 0|W^*_{-1}W^*_{-4}\neq\langle 0|W^*_{0}W^*_{0}W^*_{-1}W^*_{-4}.
\end{align*}

For any $\langle u|$ in $\widetilde{\mathcal{M}^*}$ and $|v\rangle$ in $\mathcal{M}^*$, one define $\langle u|v\rangle$ by
 \begin{align*}
 \langle u|v\rangle=\langle u||v\rangle,
 \end{align*}
 where it is assumed that $\langle 0|1|0\rangle=1$. Now we have the following orthonormality result for three distinguished sets of $\widetilde{\mathcal{M}^*}$.
\begin{Theorem}\label{th2}
For $($generalized$)$ partitions $\mu=(\mu_1,\mu_2,\dots,\mu_l)$ and $\lambda=(\lambda_1,\lambda_2,\dots,\lambda_l)$, one has that
\begin{align}
\label{e:or}&\langle \mu|\lambda\rangle=\delta_{\lambda\mu},\\
\label{e:sp3}&\langle \mu^{\rm sp}|\lambda^{\rm sp}\rangle=\delta_{\lambda\mu},\\
\label{e:o3}&\langle \mu^{\rm o}|\lambda^{\rm o}\rangle=\delta_{\lambda\mu},
\end{align}
where the partitions in \eqref{e:or} are usual ones and the partitions in \eqref{e:sp3} and \eqref{e:o3} are
generalized partitions.
\end{Theorem}
\begin{proof}
Using the commutation relations \eqref{e:com}, we have
\begin{align*}
\langle\mu|\lambda\rangle &= \langle 0|X^*_{-\mu_l}\cdots X^*_{-\mu_2}X^*_{-\mu_1}X_{-\lambda_1}X_{-\lambda_2}\cdots X_{-\lambda_l}|0\rangle\\
&= \langle 0|X^*_{-\mu_l}\cdots X^*_{-\mu_2}\big(\delta_{\mu_1,\lambda_1}-X_{-\lambda_1-1}X^*_{-\mu_1-1}\big)X_{-\lambda_2}\cdots X_{-\lambda_l}|0\rangle.
\end{align*}
The second summand is zero if $\mu_1\geq\lambda_1$, since the term $X^*_{-\mu_1-1}X_{-\lambda_2}\cdots X_{-\lambda_l}$ can be rewritten as $(-1)^{l-1}X_{-\lambda_2-1}\cdots X_{-\lambda_l-1}X^*_{-\mu_1-l}$ by \eqref{e:com}, and thus it
kills $|0\rangle$ according to \eqref{e:re1}. If $\lambda_1>\mu_1$, we move $X_{-\lambda_1-1}$ to the left
to kill $\langle 0|$. Either way shows that the inner product is simplified
to
\begin{align*}
\langle\mu|\lambda\rangle=\delta_{\lambda_1,\mu_1}\langle 0|X^*_{-\mu_l}\cdots X^*_{-\mu_2}X_{-\lambda_2}\cdots X_{-\lambda_l}|0\rangle.
\end{align*}
Subsequently $\langle\mu|\lambda\rangle=\delta_{\lambda\mu}$.

In the case of $\langle \mu^{\rm sp}|\lambda^{\rm sp}\rangle$, for $\mu_1\geq\lambda_1$ we also have that
\begin{align*}
\langle \mu^{\rm sp}|\lambda^{\rm sp}\rangle=\delta_{\mu_1,\lambda_1}\langle 0|Y^*_{-\mu_l}\cdots Y^*_{-\mu_2}Y_{-\lambda_2}\cdots Y_{-\lambda_l}|0\rangle.
\end{align*}
Now consider the subcase when $\mu_1<\lambda_1$. It follows from \eqref{e:re1} that
\begin{gather*}
\langle \mu^{\rm sp}|\lambda^{\rm sp}\rangle\\
\qquad=\begin{cases}
0, \quad \mu_1+i-1\neq \lambda_i \ \text{for all} \ 1\leq i\leq l,\\
(-1)^{i-1}\langle 0|Y^*_{-\mu_l}\cdots Y^*_{-\mu_2}Y_{-\lambda_1-1}\cdots Y_{-\lambda_{i-1}-1}Y_{-\lambda_{i+1}} \cdots Y_{-\lambda_l}|0\rangle, \quad \mu_1+i-1= \lambda_i.
\end{cases}\!
\end{gather*}
Continuing the process, $\langle \mu^{\rm sp}|\lambda^{\rm sp}\rangle$ equals to 0 or $(-1)^\epsilon\langle 0|Y^*_{-\mu_l}Y_{-\lambda_1-l+1}|0\rangle$, while
\[
\langle 0|Y^*_{-\mu_l}Y_{-\lambda_1-l+1}|0\rangle=-\langle 0|Y_{-\lambda_1-l}Y^*_{-\mu_l-1}|0\rangle=0
\]
 since $-\mu_l-1<0$ and by \eqref{e:re1}. We therefore obtain~\eqref{e:sp3} by combining the relations between~$\mu_i$ and $\lambda_i$. The relation~\eqref{e:o3} can be proved similarly.
\end{proof}

\begin{Remark}
The equation \eqref{e:sp3} makes sense only when we fix a positive integer $l$ and the length of generalized partition $\mu=(\mu_1, \dots,\mu_m,\underbrace{0,\dots,0}_{l-m})$ is not bigger than $l$, then
\begin{align*}
\langle \mu^{\rm sp}|=\langle 0|\underbrace{Y^*_0\cdots Y^*_0}_{l-m}Y^*_{-\mu_m}\cdots Y^*_{-\mu_1}.
\end{align*}
For example, it is easy to check that
\begin{align*}
\langle 0|Y^*_0 Y^*_0Y^*_{-1}Y^*_{-4}Y_{-4}Y_{-1}Y_{-1}Y_{-1}|0\rangle=0,\qquad\langle 0|Y^*_{-1}Y^*_{-4}Y_{-4}Y_{-1}Y_{-1}Y_{-1}|0\rangle=-1,
\end{align*}
thus for $\mu=(4, 1, 0, 0)$
\begin{align*}
\langle \mu^{\rm sp}|=\langle 0|Y^*_0 Y^*_0Y^*_{-1}Y^*_{-4}.
\end{align*}
The same remark applies to the orthogonal case \eqref{e:o3}.
\end{Remark}

\section[Skew symplectic/orthogonal Schur functions and Jacobi--Trudi identities]{Skew symplectic/orthogonal Schur functions\\ and Jacobi--Trudi identities}\label{s3}

This section first recalls the vertex operator construction of skew Schur functions ${\rm s}_{\lambda/\mu}(x)$.
Then we extend the method to realize skew symplectic Schur functions ${\rm sp}_{\lambda/\mu}\big(x^{\pm}\big)$ and skew orthogonal Schur functions ${\rm o}_{\lambda/\mu}\big(x^{\pm}\big)$. Their
Jacobi--Trudi formulas are also provided.

The Jacobi--Trudi identities for the Schur function ${\rm s}_\lambda(x)$, the symplectic Schur function ${\rm sp}_\lambda\big(x^{\pm}\big)$, and the orthogonal Schur function ${\rm o}_\lambda\big(x^{\pm}\big)$
can be treated in the same manner by the vertex operator. Recall the classical Jacobi--Trudi identity for the Schur functions~\cite[formula~(3.4)]{Mac1995}
\begin{align}\label{e:JT1}
{\rm s}_\nu(x)=\det(h_{\lambda_i-i+j}(x))_{1\leq i,j\leq N},
\end{align}
where the complete homogeneous function $h_n(x)$ is defined by $\prod^N_{i=1}\frac{1}{1-x_iz}=\sum_{n\in \mathbb{Z}}h_n(x)z^n$. For the symplectic and the orthogonal cases, Schur function ${\rm sp}_\lambda\big(x^{\pm}\big)$ and orthogonal Schur function ${\rm o}_\lambda\big(x^{\pm}\big)$ admit the Jacobi--Trudi formulas~\cite[Theorems~1.3.2 and~1.3.3]{KT1987} and \cite[formulas~(4.1) and~(4.2)]{JN2015}
\begin{align}
&\label{e:JT2}{\rm sp}_\nu\big(x^{\pm}\big)=\frac{1}{2}\det\big(h_{\nu_i-i+j}\big(x^{\pm}\big)+h_{\nu_{i}-i-j+2}\big(x^{\pm}\big)\big)_{1\leq i,j\leq N},\\
&\label{e:JT3}{\rm o}_\nu\big(x^{\pm}\big)=\det\big(h_{\nu_i-i+j}\big(x^{\pm}\big)-h_{\nu_{i}-i-j}\big(x^{\pm}\big)\big)_{1\leq i,j\leq N},
\end{align}
where $h_n\big(x^{\pm}\big)$ \cite{Ba1996,JN2015,Wey1946} is defined by $\prod^N_{i=1}\frac{1}{(1-x_iz)(1-x^{-1}_iz)}=\sum_{n\in \mathbb{Z}}h_n\big(x^{\pm}\big)z^n$. It is obvious that $h_n(x)=h_n\big(x^{\pm}\big)=0$ for $n<0$.

Introduce the (half) vertex operators
\begin{align}
\Gamma_+(z)=\exp\Biggl(\sum^\infty_{n=1}\frac{a_n}{n}z^n\Biggr), \qquad
\Gamma_-(z)=\exp\Biggl(\sum^\infty_{n=1}\frac{a_{-n}}{n}z^n\Biggr).\label{e:eq10}
\end{align}
For $\{x\}=(x_1,\dots,x_N)$, $\big\{x^{\pm}\big\}=\big(x_1,x^{-1}_1,\dots,x_N,x^{-1}_N\big)$, we denote
\begin{align*}
&\Gamma_+(\{x\})=\prod^{N}_{i=1}\Gamma_+(x_i),\qquad
\Gamma_+\big(\big\{x^{\pm}\big\}\big)=\prod^{N}_{i=1}\Gamma_+(x_i)\Gamma_+\big(x^{-1}_i\big),\qquad
\Gamma_-(\{x\})=\prod^{N}_{i=1}\Gamma_-(x_i).
\end{align*}

\subsection{Skew Schur functions}
For a partition $\mu=(\mu_1,\dots,\mu_l)$, one can rewrite $X^*_{-\mu_l}\cdots X^*_{-\mu_1}$ by permuting the factors~\cite{Jing2000}
and \eqref{e:com}:
\begin{align}\label{e:sc5}
\varepsilon(\sigma)X^*_{-\mu_l}\cdots X^*_{-\mu_1}=X^*_{-\mu_{\sigma(l)}+\sigma(l)-l}\cdots X^*_{-\mu_{\sigma(i)}+\sigma(i)-i}\cdots X^*_{-\mu_{\sigma(1)}+\sigma(1)-1},
\end{align}
where $\sigma\in S_l$.
The skew Schur function ${\rm s}_{\lambda/\mu}(x)$ generalizes the Schur function and has a~determinantal expression~\cite[formula~(5.4)]{Mac1995} for two partitions $\mu=(\mu_1, \dots,\mu_l)$ and $\lambda=(\lambda_1, \dots,\lambda_k)$:
\begin{align*}
{\rm s}_{\lambda/\mu}(x)=\det(h_{\lambda_i-\mu_j-i+j})_{1\leq i,j\leq k}.
\end{align*}

\begin{Proposition}\label{pro10}Let $\mu=(\mu_1, \dots,\mu_l)$ and $\lambda=(\lambda_1, \dots,\lambda_k)$ be two partitions. Then
\begin{align}\label{e:sc6}
\langle\mu|\Gamma_+(\{x\})|\lambda\rangle={\rm s}_{\lambda/\mu}(x).
\end{align}
\end{Proposition}
\begin{proof}
Using the relations
\begin{align*}
&\Gamma_+(z)=\Gamma_-\big(z^{-1}\big)X^*\big(z^{-1}\big),\\
&X^*\big(z_2^{-1}\big)\Gamma_-\big(z^{-1}_1\big)=\bigg(1-\frac{z_2}{z_1}\bigg)^{-1}\Gamma_-\big(z^{-1}_1\big)X^*\big(z_2^{-1}\big),\\
&\langle0|\Gamma_-(z)=\langle0|,
\end{align*}
we get that
\begin{align*}
\notag\langle0|\Gamma_+(\{x\})={}&\prod_{1\leq i< j\leq N}\bigg(1-\frac{x_j}{x_i}\bigg)^{-1}\langle0|X^*\big(x_N^{-1}\big)\cdots X^*\big(x_1^{-1}\big)\\
={}&\prod_{1\leq i< j\leq N}\bigg(1-\frac{x_j}{x_i}\bigg)^{-1}\sum_{n_1,\dots,n_N\in\mathbb{Z}}\langle0|X^*_{-n_N}\cdots X^*_{-n_1}x^{n_N}_N\cdots x^{n_1}_1.
\end{align*}
For 
$\nu=(\nu_1,\dots,\nu_N)$, the relation \eqref{e:sc5} says that the coefficient of $\langle\nu|=\langle 0|X^*_{-\nu_N}\cdots X^*_{-\nu_1}$ in $\langle0|\Gamma_+(\{x\})$
 is
\begin{gather*}
\prod_{1\leq i<j\leq N}\bigg(1-\frac{x_j}{x_i}\bigg)^{-1}\sum_{\sigma\in S_N}\varepsilon(\sigma)x^{\nu_{\sigma(1)}-\sigma(1)+1}_1\cdots x^{\nu_{\sigma(i)}-\sigma(i)+i}_i\cdots x^{\nu_{\sigma(N)}-\sigma(N)+N}_N\\
 \quad\qquad= \prod_{1\leq i<j\leq N} (x_i-x_j)^{-1}\sum_{\sigma\in S_N}\varepsilon(\sigma)x^{\nu_{\sigma(1)}-\sigma(1)+N}_1\cdots x^{\nu_{\sigma(i)}-\sigma(i)+N}_i\cdots x^{\nu_{\sigma(N)}-\sigma(N)+N}_N\\
 \quad\qquad= \frac{\det\big(x^{\nu_j+N-j}_i\big)_{1\leq i,j\leq N}}{\prod_{1\leq i<j\leq N} (x_i-x_j)}=\frac{\det\big(x^{\nu_j+N-j}_i\big)_{1\leq i,j\leq N}}{\det\big(x^{N-j}_i\big)_{1\leq i,j\leq N}}={\rm s}_\nu(x),
\end{gather*}
where ${\rm s}_\nu(x)$ is the Schur function \eqref{e:s50} associated with the partition $\nu$. We therefore get the expansion
\begin{align}\label{e:sc8}
\langle0|\Gamma_+(\{x\})=\sum_{\substack{l(\nu)\leq N}}\langle\nu|{\rm s}_\nu(x),
\end{align}
where the sum is over all partitions $\nu$ with length $\leq N$. It then follows from \eqref{e:sc5} and the Jacobi--Trudi formula \eqref{e:JT1} that
\begin{align*}
\langle0|\Gamma_+(\{x\})=\sum_{n_1,n_2,\dots,n_N\geq 0}\langle 0|X^*_{-n_N}\cdots X^*_{-n_1}h_{n_N}(x)\cdots h_{n_1}(x).
\end{align*}
By \eqref{e:ve1}, we have that
\begin{align*}
X^{*}(z)\Gamma_+(\{x\})=\prod^N_{i=1}\frac{1}{(1-x_iz)}\Gamma_+(\{x\})X^{*}(z)
=\sum_{i\geq 0}h_i(x)z^i\Gamma_+(\{x\})X^{*}(z),
\end{align*}
 i.e.,
\begin{align*}
X^{*}_{-n}\Gamma_+(\{x\}) = \Gamma_+(\{x\})\sum_{i\geq 0}h_i(x)X^{*}_{-n-i}.
\end{align*}
It follows from \eqref{e:sc2} that
\begin{align}
\langle\mu|\Gamma_+(\{x\})={}& \langle0|X^{*}_{-\mu_l}\cdots X^{*}_{-\mu_1}\Gamma_+(\{x\}) \nonumber\\
={}& \langle0|\Gamma_+(\{x\})\sum_{i_1,\dots,i_l\geq 0}X^{*}_{-\mu_l-i_l}\cdots X^{*}_{-\mu_1-i_1}h_{i_l}(x)\cdots h_{i_1}(x)\nonumber\\
={}& \sum_{i_{l+1},\dots,i_{l+N}\in\mathbb{Z}}\langle 0|X^*_{-i_{l+N}}\cdots X^*_{-i_{l+1}}h_{i_{l+N}}(x)\cdots h_{i_{l+1}}(x)\nonumber\\
&{}\times \sum_{i_1,\dots,i_l\geq 0}X^{*}_{-\mu_l-i_l}\cdots X^{*}_{-\mu_1-i_1}h_{i_l}(x)\cdots h_{i_1}(x)\nonumber\\
={}& \sum_{i_1,\dots,i_{l+N}\in\mathbb{Z}}\langle 0|X^*_{-i_{l+N}}\cdots X^*_{-i_{l+1}}X^{*}_{-\mu_l-i_l}\cdots X^{*}_{-\mu_1-i_1}\nonumber\\
 &{}\times h_{i_{l+N}}(x)\cdots h_{i_{l+1}}(x)h_{i_l}(x)\cdots h_{i_1}(x).\label{e:sc10}
\end{align}
For partition $\lambda$ with $l\leq l(\lambda)\leq l+N$, the relations \eqref{e:sc5} and \eqref{e:sc10} imply that the coefficient of~$\langle\lambda|$ in $\langle\mu|\Gamma_+(\{x\})$ is
\begin{align*}
\sum_{\sigma\in S_{l+N}}\epsilon(\sigma)\prod^{l+N}_{j=1}h_{\lambda_{\sigma(j)}-\mu_j-\sigma(j)+j}(x) = \det\big(h_{\lambda_i-\mu_j-i+j}(x)\big)_{1\leq i,j\leq l+N} = {\rm s}_{\lambda/\mu}(x),
\end{align*}
i.e.,
\begin{align*}
\langle\mu|\Gamma_+(\{x\}) = \sum_{\substack{l\leq l(\lambda)\leq l+N}}\langle\lambda|{\rm s}_{\lambda/\mu}(x).
\end{align*}
Subsequently (by the orthogonality \eqref{e:or})
\begin{align*}
\langle\mu|\Gamma_+(\{x\})|\lambda\rangle = {\rm s}_{\lambda/\mu}(x).
\end{align*}
This completes the proof.
\end{proof}

\begin{Remark}
Okounkov~\cite{Oko2001} derived 
\eqref{e:sc6} by 
computing $\Gamma_+(\{x\})|\lambda\rangle$. We can show \eqref{e:sc6} using the vertex operator approach (see Appendix~\ref{appA}). Though
we do not know how to generalize either method to skew-type orthogonal/symplectic cases, the method of Proposition~\ref{pro10} can be generalized to both cases.
\end{Remark}

\subsection{Skew symplectic Schur functions} It follows from \eqref{e:re2} and \eqref{e:com} that
\begin{align*}
\begin{aligned}
\langle 0|Y^*_{-n_1}\cdots Y^*_{-n_i}\cdots Y^*_{-n_N}&{}=(-1)^{i-1}\langle 0|Y^*_{-n_i-i+1}Y^*_{-n_1+1}\cdots Y^*_{-n_{i-1}+1}Y^*_{-n_{i+1}}\cdots Y^*_{-n_N}\\
&{}= (-1)^i\langle 0|Y^*_{n_i+i+1}Y^*_{-n_1+1}\cdots Y^*_{-n_{i-1}+1}Y^*_{-n_{i+1}}\cdots Y^*_{-n_N}\\
&{}= -\langle 0|Y^*_{-n_1}\cdots Y^*_{-n_{i-1}}Y^*_{n_i+2i}Y^*_{-n_{i+1}}\cdots Y^*_{-n_N}.
\end{aligned}
\end{align*}
In other words, if we replace the $i$th factor $Y^*_{-n_i}$ of $\langle 0|Y^*_{-n_1}\cdots Y^*_{-n_i}\cdots Y^*_{-n_N}$ by $-Y^*_{n_i+2i}$, it remains the same. We use the notation $(^a_b)$ to mean either of $a$ or $b$. Then for a generalized partition $\mu=(\mu_1,\dots, \mu_l)$,
\begin{align*}
\langle 0|\begin{pmatrix}{Y^*_{-\mu_l}}\\{-Y^*_{\mu_l+2}}\end{pmatrix}\cdots \begin{pmatrix}{Y^*_{-\mu_i}}\\{-Y^*_{\mu_i+2(l+1-i)}}\end{pmatrix}\cdots \begin{pmatrix}{Y^*_{-\mu_1}}\\{-Y^*_{\mu_1+2l}}\end{pmatrix}=\langle\mu^{\rm sp}|.
\end{align*}
Note that there is a similar action of $S_N$ on the vectors $\langle 0|Y^*_{-n_1}\cdots Y^*_{-n_i}\cdots Y^*_{-n_N}$
as in \eqref{e:sc5} using the commutation relation $Y^*_iY^*_j=-Y^*_{j-1}Y^*_{i+1}$ \eqref{e:com}. Therefore, we have that
\begin{align}\label{e:sp9}
\varepsilon(\sigma)\langle 0|\!\begin{pmatrix}{Y^*_{-\mu_{\sigma(l)}+\sigma(l)-l}}\\{-Y^*_{\mu_{\sigma(l)}-\sigma(l)+l+2}}\end{pmatrix}\!\cdots \! \begin{pmatrix}{Y^*_{-\mu_{\sigma(i)}+\sigma(i)-i}}\\{-Y^*_{\mu_{\sigma(i)}-\sigma(i)+2(l+1)-i}}\end{pmatrix}\!\cdots \! \begin{pmatrix}{Y^*_{-\mu_{\sigma(1)}+\sigma(1)-1}}\\{-Y^*_{\mu_{\sigma(1)}-\sigma(1)+2l+1}}\end{pmatrix}\!=\langle\mu^{\rm sp}|.\!
\end{align}

We now define a family of symmetric functions by vertex operators.
\begin{Proposition}\label{pro3}
Let $\mu=(\mu_1,\dots,\mu_l)$ and $\lambda=(\lambda_1,\dots,\lambda_{l+N})$ be two generalized partitions, and $\big\{x^{\pm}\big\}=\big\{x_1, x_1^{-1} \ldots, x_N, x_N^{-1}\big\}$. Then 
one has that
\begin{align}\label{e:sp15}
\langle\mu^{\rm sp}|\Gamma_+\big(\big\{x^{\pm}\big\}\big)|\lambda^{\rm sp}\rangle=\det(a_{ij})_{1\leq i,j\leq l+N},
\end{align}
where
\begin{equation}\label{e:eq6}
a_{ij}=\begin{cases}
h_{\lambda_i-\mu_j-i+j}\big(x^{\pm}\big), & 1\leq j\leq l+1,\\
h_{\lambda_i-i+j}\big(x^{\pm}\big)+h_{\lambda_i-i-j+2l+2}\big(x^{\pm}\big), & l+2\leq j\leq l+N.
\end{cases}
\end{equation}
\end{Proposition}

We call the symmetric function \eqref{e:sp15} as the {\it skew symplectic Schur function} $ {\rm sp}_{\lambda/\mu}\big(x^{\pm}\big)$ , i.e.,
\begin{align}\label{e:sp21}
{\rm sp}_{\lambda/\mu}\big(x^{\pm}\big)=\det(a_{ij})_{1\leq i,j\leq l+N}=\langle\mu^{\rm sp}|\Gamma_+\big(\big\{x^{\pm}\big\}\big)|\lambda^{\rm sp}\rangle.
\end{align}
We will justify this by showing these symmetric functions obey the branching rule, therefore they agree with the skew symplectic Schur functions defined by Koike and Terada~\cite{KT1990}.

\begin{proof}
By definition of the vertex operator $Y^*(z)$, one has that
\begin{align*}
\notag\langle 0|\Gamma_+\big(\big\{x^{\pm}\big\}\big)=&\prod_{1\leq i\leq j\leq N}\frac{1}{1-x_ix_j}\prod_{1\leq i<j\leq N}\frac{1}{1-\frac{x_i}{x_j}}\langle 0|Y^*(x_N)\cdots Y^*(x_1)\\
=&\frac{\prod^{N}_{i=1}x^{i-1-N}_i}{\det\big(x^{-(N-j+1)}_i-x^{N-j+1}_i\big)^N_{i,j=1}}\langle 0|Y^*(x_N)\cdots Y^*(x_1),
\end{align*}
where we have used the Vandermonde type identity~\cite[formula~(4.6)]{JN2015} and~\cite[p.~229]{Wey1946}:
\begin{align*}
\det\big(x^{j-1}_i-x^{2N-j+1}_i\big)^N_{i,j=1}=\prod_{1\leq i<j\leq N}(x_j-x_i)\prod_{1\leq i\leq j\leq N}(1-x_ix_j).
\end{align*}
In view of \eqref{e:sp9} for partition $\nu=(\nu_1,\dots,\nu_N)$, the coefficient of $\langle\nu^{\rm sp}|$ in $\langle 0|Y^*(x_N)\cdots Y^*(x_1)$ is
\begin{align*}
\sum_{\sigma\in S_N}\varepsilon(\sigma)\prod^{N}_{i=1}\big(x^{-\nu_{\sigma(i)}+\sigma(i)-i}_i-x^{\nu_{\sigma(i)}-\sigma(i)+2N+2-i}_i\big) = \det\big(x^{-\nu_j+j-i}_i-x^{\nu_j-j+2N+2-i}_i\big)^N_{i,j=1}.
\end{align*}
Then we have
\begin{align}\label{e:sp12}
\notag\langle 0|\Gamma_+\big(\big\{x^{\pm}\big\}\big)
=&\sum_{\nu=(\nu_1,\dots,\nu_N)}\langle\nu^{\rm sp}|\frac{\prod^{N}_{i=1}x^{i-1-N}_i\det\big(x^{-\nu_j+j-i}_i-x^{\nu_j-j+2N+2-i}_i\big)^N_{i,j=1}}{\det\big(x^{-(N-j+1)}_i-x^{N-j+1}_i\big)^N_{i,j=1}}\\
\notag=&\sum_{\nu=(\nu_1,\dots,\nu_N)}\langle\nu^{\rm sp}|\frac{\det\big(x^{-\nu_j-(N-j+1)}_i-x^{\nu_j+(N-j+1)}_i\big)^N_{i,j=1}}{\det\big(x^{-(N-j+1)}_i-x^{N-j+1}_i\big)^N_{i,j=1}}\\
\notag=&\sum_{\nu=(\nu_1,\dots,\nu_N)}\langle\nu^{\rm sp}|\frac{\det\big(x^{\nu_j+(N-j+1)}_i-x^{-\nu_j-(N-j+1\big)}_i)^N_{i,j=1}}{\det\big(x^{N-j+1}_i-x^{-(N-j+1)}_i\big)^N_{i,j=1}}\\
=&\sum_{\nu=(\nu_1,\dots,\nu_N)}\langle\nu^{\rm sp}|{\rm sp}_\nu\big(x^{\pm}\big),
\end{align}
where the sum is over all generalized partitions $\nu=(\nu_1,\dots,\nu_N)$ and we have used the bialternant formula of the symplectic Schur function ${\rm sp}_\nu\big(x^{\pm}\big)$ (see \eqref{e:sp50}) associated to the partition $\nu$. Invoking \eqref{e:sp9} again and also by the Jacobi--Trudi formula for the symplectic Schur functions~\eqref{e:JT2},
we obtain another expression for $\langle 0|\Gamma_+\big(\big\{x^{\pm}\big\}\big)$:
\begin{align}\label{e:sp13}
&\langle 0|\Gamma_+\big(\big\{x^{\pm}\big\}\big)\\
&\qquad=\sum_{n_1\geq -N+1,n_2\geq -N+2,\dots,n_N\geq 0}
\langle 0|Y^*_{-n_N}\cdots Y^*_{-n_1}h_{n_1}\big(x^{\pm}\big)\prod^N_{i=2}\big(h_{n_i}\big(x^{\pm}\big)+h_{n_i-2i+2}\big(x^{\pm}\big)\big).\notag
\end{align}
By definition of vertex operator $Y^*(z)$, we have
\begin{align*}
 {}Y^*(z)\Gamma_+\big(\big\{x^{\pm}\big\}\big)& =\prod^N_{i=1}\frac{1}{(1-x_iz)\big(1-x^{-1}_iz\big)}\Gamma_+\big(\big\{x^{\pm}\big\}\big)Y^*(z)\\
& =\sum_{i\geq 0}h_i\big(x^{\pm}\big)z^i\Gamma_+\big(\big\{x^{\pm}\big\}\big)Y^*(z),
\end{align*}
which implies
\begin{align*}
Y^*_{-n}\Gamma_+\big(\big\{x^{\pm}\big\}\big) = \Gamma_+\big(\big\{x^{\pm}\big\}\big)\sum_{i\geq 0}h_i\big(x^{\pm}\big)Y^*_{-n-i}.
\end{align*}
It then follows from \eqref{e:sp6}, \eqref{e:sp13} and \eqref{e:sp9} that
\begin{gather}
\notag \langle\mu^{\rm sp}|\Gamma_+\big(\big\{x^{\pm}\big\}\big)=\sum_{i_1,\dots,i_l\geq 0}h_{i_1}\big(x^{\pm}\big)\cdots h_{i_l}\big(x^{\pm}\big)\langle 0|\Gamma_+\big(\big\{x^{\pm}\big\}\big)Y^*_{-\mu_l-i_l}\cdots Y^*_{-\mu_1-i_1}
\\ \notag \hphantom{\langle\mu^{\rm sp}|\Gamma_+\big(\big\{x^{\pm}\big\}\big)}{}
=\sum_{\substack{i_1,\dots,i_l\geq 0\\i_{l+1}\geq -N+1,\dots,i_{l+N}\geq 0}}\langle 0|Y^*_{-i_{l+N}}\cdots Y^*_{-i_{l+1}}Y^*_{-\mu_l-i_l}\cdots Y^*_{-\mu_1-i_1}\prod^{l+1}_{k=1}h_{i_k}\big(x^{\pm}\big)
\\ \notag \hphantom{\langle\mu^{\rm sp}|\Gamma_+\big(\big\{x^{\pm}\big\}\big)=}{}
 \times\prod^{l+N}_{k=l+2}\big(h_{i_k}\big(x^{\pm}\big)+h_{i_k-2(k-l)+2}\big(x^{\pm}\big)\big)
\\ \notag \hphantom{\langle\mu^{\rm sp}|\Gamma_+\big(\big\{x^{\pm}\big\}\big)}{}
=\sum_{\lambda=(\lambda_1,\dots,\lambda_{l+N})}\langle\lambda^{\rm sp}|\sum_{\sigma\in S_{l+N}}
\prod^{l+1}_{k=1}h_{\lambda_{\sigma(k)}-\mu_k-\sigma(k)+k}\big(x^{\pm}\big)
\\ \notag \hphantom{\langle\mu^{\rm sp}|\Gamma_+\big(\big\{x^{\pm}\big\}\big)=}{}
 \times\prod^{l+N}_{k=l+2}\big(h_{\lambda_{\sigma(k)}-\sigma(k)+k}\big(x^{\pm}\big)+h_{\lambda_{\sigma(k)}-\sigma(k)-k+2l+2}\big(x^{\pm}\big)\big)
\\ \hphantom{\langle\mu^{\rm sp}|\Gamma_+\big(\big\{x^{\pm}\big\}\big)}{}
=\sum_{\lambda=(\lambda_1,\dots,\lambda_{l+N})}\langle\lambda^{\rm sp}|\det(a_{ij})_{1\leq i,j\leq l+N},\label{e:sp14}
\end{gather}
where $a_{ij}$ are defined in \eqref{e:eq6} and in the last two equations $\lambda_{\sigma(i)}-\sigma(i)+2(l+N+1)-i$ is greater than $l+N-i$.
\end{proof}

Using the orthogonality relation \eqref{e:sp3}, for generalized partition $\nu=(\nu_1,\dots,\nu_n)$, we have
\begin{align}
\label{e:sp24}\langle\nu^{\rm sp}|\sum_{\eta=(\eta_1,\dots,\eta_n)}|\eta^{\rm sp}\rangle\langle\eta^{\rm sp}| = \langle\nu^{\rm sp}|,\qquad
\sum_{\eta=(\eta_1,\dots,\eta_n)}|\eta^{\rm sp}\rangle\langle\eta^{\rm sp}||\nu^{\rm sp}\rangle = |\nu^{\rm sp}\rangle,
\end{align}
where the sum is over all generalized partitions $\eta=(\eta_1,\dots,\eta_n)$.
\begin{Proposition}\label{e:pro10}For $x^{\pm}=\big(x_1,x^{-1}_1,\dots,x_{n-k},x^{-1}_{n-k}\big)$ and $y^{\pm}=\big(y_1,y^{-1}_1,\dots,y_{k},y^{-1}_{k}\big)$, the functions ${\rm sp}_{\lambda/\mu}$ satisfy the general branching rule
\begin{align}\label{e:sp30}
{\rm sp}_\lambda\big(x^{\pm};y^{\pm}\big)=\sum_{\mu=(\mu_1,\dots,\mu_{n-k})}{\rm sp}_\mu\big(x^{\pm}\big){\rm sp}_{\lambda/\mu}\big(y^{\pm}\big).
\end{align}
\end{Proposition}
\begin{proof}
It follows from \eqref{e:sp15} that
\begin{align}\label{e:sp25}
\langle0|\Gamma_{+}\big(x^{\pm};y^{\pm}\big)|\lambda^{\rm sp}\rangle={\rm sp}_\lambda\big(x^{\pm};y^{\pm}\big).
\end{align}
By \eqref{e:sp24},
\begin{align}
\langle0|\Gamma_{+}\big(x^{\pm};y^{\pm}\big)|\lambda^{\rm sp}\rangle&=\langle0|\Gamma_{+}\big(x^{\pm}\big)\sum_{\mu=(\mu_1,\dots,\mu_{n-k})}|\mu^{\rm sp}\rangle\langle\mu^{\rm sp}|\Gamma_{+}\big(y^{\pm}\big)|\lambda^{\rm sp}\rangle\nonumber\\
&=\sum_{\mu=(\mu_1,\dots,\mu_{n-k})}{\rm sp}_\mu\big(x^{\pm}\big){\rm sp}_{\lambda/\mu}\big(y^{\pm}\big).\label{e:sp2}
\end{align}
Combining \eqref{e:sp25} and \eqref{e:sp2}, we obtain the result.
\end{proof}

\begin{Remark} Since ${\rm sp}_{\lambda/\mu}$ satisfies the general branching rule \eqref{e:sp30} and reduces to the symplectic Schur function when $\mu=\varnothing$, it
is rightfully called
the skew symplectic Schur function. For $k=1$, the general branching rule \eqref{e:sp30} reduces to the branching rule for Koike--Terada's symplectic Schur functions~\cite[Theorem~3.1]{KT1990}. In fact, ${\rm sp}_{\lambda/\mu}$ agrees with that defined by Koike and Terada by using the forthcoming Gelfand--Tsetlin patterns \eqref{e:sp19}.
\end{Remark}
\begin{Remark}
We stress that the partitions $\mu$ of the skew symplectic Schur function ${\rm sp}_{\lambda/\mu}$
in \eqref{e:sp15} and~\eqref{e:sp21} are general partitions, i.e., some parts can be zeros. For instance, the symmetric function $\langle0|Y^*_{-1}\Gamma_+\big(\big\{x^{\pm}\big\}\big)Y_{-2}Y_{-1}Y_{-1}|0\rangle$ is the skew symplectic Schur function ${\rm sp}_{(2,1,1)/(1)}\big(x^{\pm}\big)$, and $\langle0|Y^*_{0}Y^*_{-1}\Gamma_+\big(\big\{x^{\pm}\big\}\big)Y_{-2}Y_{-1}Y_{-1}|0\rangle$ gives different expression ${\rm sp}_{(2,1,1,0)/(1,0)}\big(x^{\pm}\big)$.
The same statement applies to the skew orthogonal case \eqref{e:sk5}.
\end{Remark}
\begin{Remark}\label{e:remark1}
For two generalized partitions $\mu=(\mu_1,\dots,\mu_l)$ and $\lambda=\big(\lambda_1,\dots,\lambda_{l+1},0^{N-1}\big)$, the
formula \eqref{e:sp21} implies that
${\rm sp}_{\lambda/\mu}\big(x^{\pm}\big)={\rm s}_{\lambda/\mu}\big(x^{\pm}\big)$ since the matrix $(a_{ij})$ is a block upper-unipotent one.
In particular, if $l=N$ or $l=N-1$, the skew symplectic Schur function ${\rm sp}_{(\lambda_1,\dots,\lambda_N,0^l)/(0^l)}\big(x^{\pm}\big)$ is equal to the Schur function ${\rm s}_{(\lambda_1,\dots,\lambda_N)}\big(x^{\pm}\big)$ for
\[
\big\{x^{\pm}\big\}=\big(x_1,x^{-1}_1,\dots,x_N,x^{-1}_N\big).
\]
One can also get the special cases of skew symplectic Schur functions via Gelfand--Tsetlin patterns in the following.
\end{Remark}

\subsection{Skew orthogonal Schur functions} We can treat the orthogonal Schur functions similarly. Using the commutation relations
\begin{align*}
\langle 0|W^*_n=\langle 0|W^*_{-n},\qquad W^*_iW^*_j=-W^*_{j-1}W^*_{i+1},
\end{align*}
we see that the partition elements \eqref{e:o6} with partitions $\mu=(\mu_1,\dots,\mu_l)$ can be expressed as
\begin{gather}\label{e:o9}
\varepsilon(\sigma)\langle 0|\begin{pmatrix}{W^*_{-\mu_{\sigma(l)}+\sigma(l)-l}}\\{\delta_lW^*_{\mu_{\sigma(l)}-\sigma(l)+l}}\end{pmatrix}\cdots \begin{pmatrix}{W^*_{-\mu_{\sigma(i)}+\sigma(i)-i}}\\{\delta_iW^*_{\mu_{\sigma(i)}-\sigma(i)+2l-i}}\end{pmatrix}\cdots \begin{pmatrix}{W^*_{-\mu_{\sigma(1)}+\sigma(1)-1}}\\{\delta_1W^*_{\mu_{\sigma(1)}-\sigma(1)+2l-1}}\end{pmatrix}=\langle\mu^{\rm o}|,
\end{gather}
where $\delta_i$ denotes $\delta_{\mu_l\neq 0}\delta_{\sigma(i)\neq l}$.

Recall the Vandermonde type identity~\cite[formula~(4.4)]{JN2015} and \cite[p.~221]{Wey1946} in this case
\begin{align*}
\det\big(x^{N-j}_i+x^{N+j-2}_i\big)^N_{i,j=1}=2\prod_{1\leq i<j\leq N} (x_i-x_j)(1-x_ix_j).
\end{align*}
Therefore,
\begin{align}
\notag&\langle 0|\Gamma_+\big(\big\{x^{\pm}\big\}\big)\\
\notag&\qquad=\sum_{\nu=(\nu_1,\dots,\nu_N)}\langle\nu^{\rm o}|2(-1)^{\frac{N(N-1)}{2}}\frac{\prod^N_{i=1}x^{i-N}_i\det\big(x^{-\nu_j+j-i}_i+\delta_{j\neq N}\delta_{\nu_N\neq 0}x^{\nu_j-j+2N-i}_i\big)^N_{i,j=1}}{\det\big(x^{-j+1}_i+x^{j-1}_i\big)^N_{i,j=1}}\\
\notag&\qquad=\sum_{\nu=(\nu_1,\dots,\nu_N)}\langle\nu^{\rm o}|2(-1)^{\frac{N(N-1)}{2}}\frac{\det\big(x^{-\nu_j+j-N}_i+\delta_{j\neq N}\delta_{\nu_N\neq 0}x^{\nu_j-j+N}_i\big)^N_{i,j=1}}{\det\big(x^{-j+1}_i+x^{j-1}_i\big)^N_{i,j=1}}\\
\label{e:o11}&\qquad=\sum_{\nu=(\nu_1,\dots,\nu_N)}\langle\nu^{\rm o}|{\rm o}_\nu\big(x^{\pm}\big),
\end{align}
where the sum is over all partitions $\nu=(\nu_1,\dots,\nu_N)$ and ${\rm o}_\nu\big(x^{\pm}\big)$ is the orthogonal Schur function associated to partition $\nu$ (see \eqref{e:o50}). It follows from \eqref{e:o9} and the Jacobi--Trudi formula \eqref{e:JT3} that
\begin{align}
\notag&\langle 0|\Gamma_+\big(\big\{x^{\pm}\big\}\big)\\
&\quad\qquad{}=\sum_{n_1\geq -N+1,n_2\geq -N+2,\dots,n_N\geq 0}
\langle 0|W^*_{-n_N}\cdots W^*_{-n_1}\prod^N_{i=1}\big(h_{n_i}\big(x^{\pm}\big)-h_{n_i-2i}\big(x^{\pm}\big)\big).\label{e:o10}
\end{align}

The following proposition can be proved by the same method of Proposition~\ref{pro3}, \eqref{e:o9}, and~\eqref{e:o10}.
\begin{Proposition}
For generalized partitions $\mu=(\mu_1,\dots,\mu_l)$ and $\lambda=(\lambda_1,\dots,\lambda_{l+N})$, and $\big\{x^{\pm}\big\}=\big\{x_1, x_1^{-1}, \ldots, x_N, x_N^{-1}\big\}$, one has that
\begin{align}
\label{e:sk5}\langle\mu^{\rm o}|\Gamma_+\big(\big\{x^{\pm}\big\}\big)|\lambda^{\rm o}\rangle = \det(b_{ij})^{l+N}_{i,j=1},
\end{align}
where
\begin{align*}
b_{ij}=\begin{cases}
h_{\lambda_i-\mu_j-i+j}\big(x^{\pm}\big), &1\leq j\leq l,\\
h_{\lambda_i-i+j}\big(x^{\pm}\big)-h_{\lambda_i-i-j+2l}\big(x^{\pm}\big), & l+1\leq j\leq l+N.
\end{cases}
\end{align*}
\end{Proposition}

We also call the symmetric function in \eqref{e:sk5} the {\it skew Schur orthogonal function} ${\rm o}_{\lambda/\mu}(x)$. In the following, we will justify the definition.

Using the orthogonality relation \eqref{e:o3}, for generalized partition $\nu=(\nu_1,\dots,\nu_n)$, one has
\begin{align}
\label{e:o24}\langle\nu^{\rm o}|\sum_{\eta=(\eta_1,\dots,\eta_n)}|\eta^{\rm o}\rangle\langle\eta^{\rm o}| = \langle\nu^{\rm o}|,\qquad \sum_{\eta=(\eta_1,\dots,\eta_n)}|\eta^{\rm o}\rangle\langle\eta^{\rm o}||\nu^{\rm o}\rangle = |\nu^{\rm o}\rangle,
\end{align}
where the sum is over all generalized partitions $\eta=(\eta_1,\dots,\eta_n)$. Similar to the proof of Proposition~\ref{e:pro10}, we have the following statement.

\begin{Proposition}For $x^{\pm}=\big(x_1,x^{-1}_1,\dots,x_{n-k},x^{-1}_{n-k}\big)$ and $y^{\pm}=\big(y_1,y^{-1}_1,\dots,y_{k},y^{-1}_{k}\big)$, the functions ${\rm o}_{\lambda/\mu}$ satisfy the general branching rule
\begin{align}\label{e:o30}
{\rm o}_\lambda\big(x^{\pm};y^{\pm}\big)=\sum_{\mu=(\mu_1,\dots,\mu_{n-k})}o_\mu\big(x^{\pm}\big){\rm o}_{\lambda/\mu}\big(y^{\pm}\big).
\end{align}
\end{Proposition}
\begin{Remark}
Based on the general branching rule for orthogonal Schur function \eqref{e:o30}, the function ${\rm o}_{\lambda/\mu}$ \eqref{e:sk5} is rightfully called the skew orthogonal Schur function. For $k=1$, the general branching rule \eqref{e:o30} reduces to the branching rule for the orthogonal Schur function~\mbox{\cite[Theorem 3.2]{KT1990}}.
\end{Remark}

\section[Gelfand--Tsetlin pattern representations for three skew-type functions]{Gelfand--Tsetlin pattern representations\\ for three skew-type functions}\label{s4}

In this section, we first derive the formulas for skew Schur functions in Gelfand--Tsetlin patterns using vertex operators. We then do the same for
skew symplectic/orthogonal Schur functions.

We start by recalling a special case of skew Schur functions in terms of vertex operators~\mbox{\cite[Proposition~2.4]{JL2020}}. We say a (generalized) partition
$\lambda$ interlaces another (generalized) partition $\nu$, denoted as $\nu\prec\lambda$, if $\lambda_{i}\ge\nu_{i}\geq\lambda_{i+1}$ for all $i$.
\begin{Lemma}[{\cite[p.~72]{Mac1995}}]\label{le3} For two interlacing partitions $\nu\prec\lambda$, $\nu=(\nu_{1},\dots,\nu_{l})$ and $\lambda=(\lambda_{1},\dots,\lambda_{l+1})$, one has that
\begin{align*}
{\rm s}_{\lambda/\nu}(t)=t^{|\lambda|-|\nu|}.
\end{align*}
\end{Lemma}
It follows from Lemma~\ref{le3} that skew Schur functions can be written as a sum over Gelfand--Tsetlin patterns: for partitions $\mu=(\mu_1,\dots,\mu_l)\subset \lambda=(\lambda_1,\dots,\lambda_{l+N})$
\begin{align*}
{\rm s}_{\lambda/\mu}(x)=\sum_{\mu=z_0\prec z_1\prec\dots\prec z_N=\lambda}\prod^N_{i=1}x^{|z_i|-|z_{i-1}|}_i,
\end{align*}
summed over all sequences of partitions $\mu=z_0\prec z_1\prec\dots\prec z_N=\lambda$, where
$z_i=(z_{i,1},z_{i,2},\dots,\allowbreak z_{i,l+i})$ satisfy
 the {\it Gelfand--Tsetlin pattern}:
\begin{align*}
z_{i+1,j}\leq z_{i,j-1}\leq z_{i+1,j-1}\qquad \text{for}\quad 2\leq j\leq l+i+1.
\end{align*}

Using the similar method as in~\cite{JL2020}, we can get the Gelfand--Tsetlin pattern formulas for skew symplectic (orthogonal) Schur functions ${\rm sp}_{\lambda/\mu}\big(x^{\pm}\big)$ \big(${\rm o}_{\lambda/\mu}\big(x^{\pm}\big)$\big). We start by proving three lemmas.
\begin{Lemma}For any integer $m$ and generalized partitions $\alpha=(\alpha_1,\dots,\alpha_k)$, $\beta=(\beta_1,\dots,\beta_{k+2})$, one has
\begin{align*}
\langle\alpha^{\rm sp}|\sum_{i\leq m,\, j\leq m+1}Y^*_{i}Y^*_{j}t^{i+j}|\beta^{\rm sp}\rangle=0.
\end{align*}
\end{Lemma}
\begin{proof}
From \eqref{e:re1} and \eqref{e:com}, we know that $Y^*_{j}|\beta^{\rm sp}\rangle=0$ for $j<-\beta_1$. Therefore,
\begin{align*}
\sum_{i\leq m,\, j\leq m+1}Y^*_{i}Y^*_{j}t^{i+j}|\beta^{\rm sp}\rangle
&=\sum_{-\beta_1-1\leq i\leq m,\, -\beta_1\leq j\leq m+1}Y^*_{i}Y^*_{j}t^{i+j}|\beta^{\rm sp}\rangle\\
&=-\sum_{-\beta_1-1\leq i\leq m,\, -\beta_1\leq j\leq m+1}Y^*_{j-1}Y^*_{i+1}t^{i+j}|\beta^{\rm sp}\rangle\\
&=-\sum_{-\beta_1-1\leq a\leq m,\, -\beta_1\leq b\leq m+1}Y^*_{a}Y^*_{b}t^{a+b}|\beta^{\rm sp}\rangle,
\end{align*}
therefore $\sum_{i\leq m,\, j\leq m+1}Y^*_{i}Y^*_{j}t^{i+j}|\beta^{\rm sp}\rangle=0$, which completes the proof.
\end{proof}

\begin{Corollary}\label{cor1}
For any integer $m$ and generalized partition $\alpha=(\alpha_1,\dots,\alpha_k)$,
\begin{align*}
\langle\alpha^{\rm sp}|\sum_{i\leq m,\, j\leq m+1}Y^*_{i}Y^*_{j}t^{i+j}=0.
\end{align*}
\end{Corollary}
\begin{Lemma}\label{le4}
For generalized partitions $\nu=(\nu_{1},\dots,\nu_{l})$ and $\lambda=(\lambda_{1},\dots,\lambda_{l+1})$, one has
\begin{align*}
{\rm sp}_{\lambda/\nu}\big(t^{\pm}\big)=\sum_{\nu\prec\alpha\prec\lambda}t^{2|\alpha|-|\lambda|-|\nu|}
\end{align*}
summed over all generalized partitions $\alpha$ such that $\nu\prec\alpha\prec\lambda$.
\end{Lemma}
\begin{proof}It follows from equation \eqref{e:sp14} that
\begin{align}\label{e:sp10}
\langle\nu^{\rm sp}|\Gamma_+\big(t^{\pm}\big)
=\sum_{i,i_1,\dots,i_l\geq 0}h_i\big(t^{\pm}\big)\prod^l_{j=1}h_{i_j}\big(t^{\pm}\big)\langle 0|Y^*_{-i}Y^*_{-\nu_l-i_l}\cdots Y^*_{-\nu_1-i_1},
\end{align}
where $h_i\big(t^{\pm}\big)=t^{-i}+t^{-i+2}+\dots+t^{i-2}+t^i=\frac{t^{-i}-t^{i+2}}{1-t^2}$.
Since $Y^*_{j-1}Y^*_{j}=0$ and Corollary~\ref{cor1}, we have
\begin{align*}
&\sum_{i\geq \nu_l+1,j\geq \nu_l}\langle 0|Y^*_{-i}Y^*_{-j}\bigl(t^{-i}-t^{i+2}\bigr)\bigl(t^{\nu_l-j}-t^{j-\nu_l+2}\bigr)\\
&\qquad\quad{}=\sum_{i\geq \nu_l+1,j\geq \nu_l}\langle 0|Y^*_{-i}Y^*_{-j}\bigl(-t^{i+2+\nu_l-j}-t^{j-\nu_l+2-i}\bigr)\\
&\qquad\quad{}=\sum_{j\geq i\geq \nu_l+1}\langle 0|Y^*_{-i}Y^*_{-j}\bigl(-t^{i+2+\nu_l-j}-t^{j-\nu_l+2-i}\bigr)\\
&\qquad\qquad{}+\sum_{i-1> j\geq\nu_l}\langle 0|Y^*_{-i}Y^*_{-j}\bigl(-t^{i+2+\nu_l-j}-t^{j-\nu_l+2-i}\bigr)\\
&\qquad\quad{}=\sum_{j\geq i\geq \nu_l+1}\langle 0|Y^*_{-i}Y^*_{-j}\bigl(-t^{i+2+\nu_l-j}-t^{j-\nu_l+2-i}\bigr)\\
&\qquad\qquad{}+\sum_{i-1> j\geq\nu_l}\langle 0|Y^*_{-j-1}Y^*_{-i+1}\bigl(t^{i+2+\nu_l-j}+t^{j-\nu_l+2-i}\bigr)\\
&\qquad\quad{}=\sum_{j\geq i\geq \nu_l+1}\langle 0|Y^*_{-i}Y^*_{-j}\bigl(-t^{i+2+\nu_l-j}-t^{j-\nu_l+2-i}\bigr)\\
&\qquad\qquad{}+\sum_{n\geq m\geq \nu_l+1}\langle 0|Y^*_{-m}Y^*_{-n}\bigl(t^{n+4+\nu_l-m}+t^{m-\nu_l-n}\bigr)\\
&\qquad\quad{}=\sum_{j\geq i\geq \nu_l+1}\langle 0|Y^*_{-i}Y^*_{-j}\bigl(t^{j+4+\nu_l-i}+t^{i-\nu_l-j}-t^{i+2+\nu_l-j}-t^{j-\nu_l+2-i}\bigr)\\
&\qquad\quad{}=\sum_{j\geq i\geq \nu_l+1}\langle 0|Y^*_{-i}Y^*_{-j}\bigl(t^{i-j}-t^{j-i+2}\bigr)\bigl(t^{-\nu_l}-t^{2+\nu_l}\bigr),
\end{align*}
where we have split $i\geq \nu_l+1$, $j\geq \nu_l$ into $i-j=1$, $i-j>1$ and $i-j\leq 0$, in other words, $i=j+1$, $i-1> j\geq\nu_l$, $j\geq i\geq \nu_l+1$.
Thus
\begin{align*}
\notag&\sum_{i,i_l\geq 0}h_i\big(t^{\pm}\big)h_{i_l}\big(t^{\pm}\big)\langle 0|Y^*_{-i}Y^*_{-\nu_l-i_l}\\
\notag&\quad{}=\frac{1}{\big(1-t^2\big)^2}\sum_{i,i_l\geq 0}\langle 0|Y^*_{-i}Y^*_{-\nu_l-i_l}\big(t^{-i}-t^{i+2}\big)\big(t^{-i_l}-t^{i_l+2}\big)\\
\notag&\quad{}=\sum_{0\leq i\leq \nu_l,\,i_l\geq 0}\langle 0|Y^*_{-i}Y^*_{-\nu_l-i_l}h_i\big(t^{\pm}\big)h_{i_l}\big(t^{\pm}\big)\\ \notag
&\qquad{}+\frac{1}{\big(1-t^2\big)^2}\sum_{i\geq \nu_l+1,\,i_l\geq 0}\langle 0|Y^*_{-i}Y^*_{-\nu_l-i_l}\big(t^{-i}-t^{i+2}\big)\big(t^{-i_l}-t^{i_l+2}\big)\\
\notag&\quad{}=\sum_{0\leq i\leq \nu_l,\,i_l\geq 0}\langle 0|Y^*_{-i}Y^*_{-\nu_l-i_l}h_i\big(t^{\pm}\big)h_{i_l}\big(t^{\pm}\big)+\sum_{\nu_l+i_l\geq i\geq \nu_l+1}\langle 0|Y^*_{-i}Y^*_{-\nu_l-i_l}h_{\nu_l}\big(t^{\pm}\big)h_{\nu_l+i_l-i}\big(t^{\pm}\big).
\end{align*}
Setting $\gamma_1=\nu_l+i_l\geq \gamma_2=i\geq 0$, then the above equation becomes
\begin{align*}
&\sum_{i,i_l\geq 0}h_i\big(t^{\pm}\big)h_{i_l}\big(t^{\pm}\big)\langle 0|Y^*_{-i}Y^*_{-\nu_l-i_l}\\
&\qquad\quad{}=\sum_{\max\{\gamma_2,\nu_l\}\leq \gamma_1}\langle 0|Y^*_{-\gamma_2}Y^*_{-\gamma_1}h_{\min\{\gamma_2,\nu_l\}}\big(t^{\pm}\big)h_{\gamma_1-\max\{\gamma_2,\nu_l\}}\big(t^{\pm}\big).
\end{align*}
Continuing the process, equation \eqref{e:sp10} becomes
\begin{gather*}
 \langle\nu^{\rm sp}|\Gamma_+\big(t^{\pm}\big)\\
 \qquad\quad{}=\sum_{\lambda}\langle\lambda^{\rm sp}|h_{\min\{\lambda_{l+1},\nu_l\}}\big(t^{\pm}\big)\Bigg(\prod^{l}_{j=2}h_{\min\{\lambda_j,\nu_{j-1}\}-\max\{\lambda_{j+1},\nu_j\}}\big(t^{\pm}\big)\Bigg)h_{\lambda_1-\max\{\lambda_2,\nu_1\}}\big(t^{\pm}\big)\\
 \qquad\quad{}=\sum_{\lambda}\langle\lambda^{\rm sp}|h_{\min\{\lambda_{l+1},\nu_l\}}\big(t^{\pm}\big)t^{\min\{\lambda_{l+1},\nu_l\}}\Bigg(\prod^{l}_{j=2}h_{\min\{\lambda_j,\nu_{j-1}\}-\max\{\lambda_{j+1},\nu_j\}}\big(t^{\pm}\big)\\
 \qquad\quad\hphantom{=\sum_{\lambda}}{}
\times t^{\min\{\lambda_j,\nu_{j-1}\}+\max\{\lambda_{j+1},\nu_j\}}\Bigg)h_{\lambda_1-\max\{\lambda_2,\nu_1\}}\big(t^{\pm}\big)t^{\lambda_1+\max\{\lambda_2,\nu_1\}} t^{-|\lambda|-|\nu|}\\ \notag
 \qquad\quad{}=\sum_{\lambda}\langle\lambda^{\rm sp}|\sum_{\nu\prec\alpha\prec\lambda}t^{2|\alpha|-|\lambda|-|\nu|},
\end{gather*}
 where we use the fact that \[
 |\lambda|+|\nu|=\min\{\lambda_{l+1},\nu_l\}+\sum^l_{j=2}\min\{\lambda_j,\nu_{j-1}\}+\sum^l_{j=2}\max\{\lambda_{j+1},\nu_j\}+\lambda_1+\max\{\lambda_2,\nu_1\}
 \] and the notations
\begin{gather*}
 0\leq \alpha_{l+1}\leq\min\{\lambda_{l+1},\nu_l\}, \qquad \max\{\lambda_{j+1},\nu_j\}\leq \alpha_{j}\leq\min\{\lambda_j,\nu_{j-1}\},\\ \max\{\lambda_2,\nu_1\}\leq\alpha_1\leq \lambda_1.
\end{gather*}
Thus the proof is completed.
\end{proof}

\begin{Lemma}\label{le2}
For generalized partitions $\nu=(\nu_{1},\dots,\nu_{l})$ and $\lambda=(\lambda_{1},\dots,\lambda_{l+1})$, one has
\begin{align*}
{\rm o}_{\lambda/\nu}\big(t^{\pm}\big)=\sum_{\nu\prec\alpha\prec\lambda}(1+\delta_{\lambda_{l+1}>0}\delta_{\nu_{l},0})t^{2|\alpha|-|\lambda|-|\nu|}
\end{align*}
summed over all partitions $\alpha$ such that $\nu\prec\alpha\prec\lambda$ and $\alpha_{l+1}\in \{0,\min\{\lambda_{l+1},\nu_{l}\}\}$.
\end{Lemma}
\begin{proof}
Using the same method for Lemma~\ref{le4}, we have
\begin{gather*}
 \langle\nu^{\rm o}|\Gamma_+\big(t^{\pm}\big) =\sum_{\lambda}\langle\lambda^{\rm o}|\big(t^{-\min\{\lambda_{l+1},\nu_l\}}+t^{\min\{\lambda_{l+1},\nu_l\}}\big)\\
 \quad{}\times\Bigg(\prod^{l}_{j=2}h_{\min\{\lambda_j,\nu_{j-1}\}-\max\{\lambda_{j+1},\nu_j\}}\big(t^{\pm}\big)\Bigg)
 h_{\lambda_1-\max\{\lambda_2,\nu_1\}}\big(t^{\pm}\big)
\\
 = \sum_{\lambda}\langle\lambda^{\rm o}|\big(1+t^{2\min\{\lambda_{l+1},\nu_l\}}\big) \Bigg(\prod^{l}_{j=2}h_{\min\{\lambda_j,\nu_{j-1}\}-\max\{\lambda_{j+1},\nu_j\}}\big(t^{\pm}\big)t^{\min\{\lambda_j,\nu_{j-1}\}+\max\{\lambda_{j+1},\nu_j\}}\Bigg)\\
 \quad\times h_{\lambda_1-\max\{\lambda_2,\nu_1\}}\big(t^{\pm}\big)t^{\lambda_1+\max\{\lambda_2,\nu_1\}} t^{-|\lambda|-|\nu|}\\
= \sum_{\lambda}\langle\lambda^{\rm o}|\sum_{\nu\prec\alpha\prec\lambda}\big(1+\delta_{\lambda_{l+1}>0}\delta_{\nu_{l},0}\big)t^{2|\alpha|-|\lambda|-|\nu|},
\end{gather*}
summed over $\alpha$ such that $\nu\prec\alpha\prec\lambda$ and $\alpha_{l+1}\in \{0,\min\{\lambda_{l+1},\nu_{l}\}\}$.
\end{proof}

The following theorem directly follows from Lemmas~\ref{le4} and~\ref{le2}.
\begin{Theorem}
For generalized partitions $\mu=(\mu_1,\dots,\mu_l)\subset \lambda=(\lambda_1,\dots,\lambda_{l+N})$, one has
\begin{align}
\label{e:sp19}&{\rm sp}_{\lambda/\mu}\big(x^{\pm}\big) = \sum_{\mu=z_0\prec z_1\prec\dots\prec z_{2N}=\lambda}\prod^N_{i=1}x^{2|z_{2i-1}|-|z_{2i}|-|z_{2i-2}|}_i,\\
\label{e:o19}&{\rm o}_{\lambda/\mu}\big(x^{\pm}\big) = \sum_{\mu=z^{\prime}_0\prec z^{\prime}_1\prec\dots\prec z^{\prime}_{2N}=\lambda}\prod^N_{i=1}\big(1+\delta_{z^{\prime}_{2i,l+i}>0}\delta_{z^{\prime}_{2i-2,l+i-1},0}\big)x^{2|z^{\prime}_{2i-1}| -|z^{\prime}_{2i}|-|z^{\prime}_{2i-2}|}_i,
\end{align}
where generalized partitions $z_k=\big(z_{k,1},\dots,z_{k,l+\lceil\frac{k}{2}\rceil}\big)$ satisfy the symplectic Gelfand--Tsetlin pattern
\begin{align*}
z_{k+1,j}\leq z_{k,j-1}\leq z_{k+1,j-1}\qquad\text{for}\quad 2\leq j\leq l+\biggl\lceil \frac{k+1}{2}\biggr\rceil
\end{align*}
and generalized partitions $z^{\prime}_k=\big(z^{\prime}_{k,1},\dots,z^{\prime}_{k,l+\lceil\frac{ k}{2}\rceil}\big)$ satisfy the orthogonal Gelfand--Tsetlin pattern
\begin{align*}
z^{\prime}_{k+1,j}\leq z^{\prime}_{k,j-1}\leq z^{\prime}_{k+1,j-1}\qquad \text{for}\quad 2\leq j\leq l+\biggl\lceil\frac{k+1}{2}\biggr\rceil
\end{align*}
subject to $z^{\prime}_{2i-1,l+i}\in\big\{0,\min\big\{z^{\prime}_{2i,l+i},z^{\prime}_{2i-2,l+i-1}\big\}\big\}$.
\end{Theorem}
\begin{Remark}
One can find an equivalent statement of \eqref{e:sp19} in~\cite{AF2020}, which defined the symplectic Gelfand--Tsetlin pattern by induction. For $\mu=\varnothing$, \eqref{e:sp19} recovers the classical Gelfand--Tsetlin patterns for symplectic/orthogonal Schur functions \cite[p.~310]{Pro1994}, and \eqref{e:o19} reduces to the multi-orthogonal Gelfand--Tsetlin pattern~\cite[p.~324]{Pro1994}.
\end{Remark}

\section{Cauchy-type identities and general branching rules}\label{s5}

In this section, we give a new proof of Cauchy identities for (symplectic/orthogonal) Schur functions by the method of vertex operators. This approach is logically independent of the Jacobi--Trudi formula. The method can be further extended to
skew Schur functions as well as skew symplectic/orthogonal Schur functions.
\begin{Lemma}\label{le5}For $x=(x_1,x_2,\dots,x_N)$, we have
\begin{align}
&\Gamma_-(\{x\})|0\rangle=\sum_{\lambda}{\rm s}_\lambda(x)|\lambda\rangle, \label{e:sc11}\\
&\Gamma_-(\{x\})|0\rangle=\prod_{1\leq i<j\leq N}(1-x_ix_j)^{-1}\sum_{\lambda}{\rm s}_\lambda(x)|\lambda^{\rm sp}\rangle, \nonumber\\ 
&\Gamma_-(\{x\})|0\rangle=\prod_{1\leq i\leq j\leq N}(1-x_ix_j)^{-1}\sum_{\lambda}{\rm s}_\lambda(x)|\lambda^{\rm o}\rangle,\nonumber 
\end{align}
where the sum is over all partitions $\lambda$ with $l(\lambda)\leq N$.
\end{Lemma}
\begin{proof}
We first prove \eqref{e:sc11} by a similar method as in Proposition~\ref{pro10}, which is different from that of~\cite{JL2020}.
Using the relations
\[
X(z)\Gamma_+\big(z^{-1}\big)=\Gamma_-(z) \qquad \text{and}\qquad \Gamma_+\big(z^{-1}_1\big)\Gamma_-(z_2)=\left(1-\frac{z_2}{z_1}\right)^{-1}\Gamma_-(z_2)\Gamma_+\big(z^{-1}_1\big),
\]
one has
\begin{align}
\Gamma_-(\{x\})|0\rangle&{}=\prod_{1\leq i<j\leq N}\bigg(1-\frac{x_j}{x_i}\bigg)^{-1}X(x_1)\cdots X(x_N)|0\rangle\notag\\
&{}=\prod_{1\leq i<j\leq N}\bigg(1-\frac{x_j}{x_i}\bigg)^{-1}\sum_{n_1,\dots,n_N}x^{n_1}_1\cdots x^{n_N}_NX_{-n_1}\cdots X_{-n_N}|0\rangle.\label{e:sc1}
\end{align}
 Using $X_iX_j=-X_{j+1}X_{i-1}$, the factors can be shuffled into the following expression:
\begin{align}\label{e:sc15}
\varepsilon(\sigma)X_{-\lambda_1}\cdots X_{-\lambda_N}=X_{-\lambda_{\sigma(1)}+\sigma(1)-1}\cdots X_{-\lambda_{\sigma(i)}+\sigma(i)-i}\cdots X_{-\lambda_{\sigma(N)}+\sigma(N)-N},
\end{align}
where $\sigma\in S_N$. Thus the coefficient of $|\lambda\rangle=X_{-\lambda_1}\cdots X_{-\lambda_N}|0\rangle$ in $\Gamma_-(\{x\})|0\rangle$ \eqref{e:sc1} is
\begin{align*}
&\prod_{1\leq i<j\leq N}\bigg(1-\frac{x_j}{x_i}\bigg)^{-1}\sum_{\sigma}\varepsilon(\sigma)x^{\lambda_{\sigma(1)}-\sigma(1)+1}_1\cdots x^{\lambda_{\sigma(i)}-\sigma(i)+i}_i\cdots x^{\lambda_{\sigma(N)}-\sigma(N)+N}_N\\
&\qquad\quad{}=\prod_{1\leq i<j\leq N} (x_i-x_j)^{-1}\sum_{\sigma}\varepsilon(\sigma)x^{\lambda_{\sigma(1)}-\sigma(1)+N}_1\cdots x^{\lambda_{\sigma(i)}-\sigma(i)+N}_i\cdots x^{\lambda_{\sigma(N)}-\sigma(N)+N}_N\\
&\qquad\quad{}=\frac{\det\big(x^{\lambda_j+N-j}_i\big)_{1\leq i,j\leq N}}{\prod_{1\leq i<j\leq N} (x_i-x_j)}
 = \frac{\det\big(x^{\lambda_j+N-j}_i\big)_{1\leq i,j\leq N}}{\det\big(x^{N-j}_i\big)_{1\leq i,j\leq N}} = {\rm s}_\lambda(x).
\end{align*}
Then the relation \eqref{e:or} gives $\langle\lambda|\Gamma_-(\{x\})|0\rangle={\rm s}_\lambda(x)$ for any $\lambda$, subsequently
\begin{align*}
\Gamma_-(\{x\})|0\rangle=\sum_{\substack{\lambda\\l(\lambda)\leq N}}{\rm s}_\lambda(x)|\lambda\rangle.
\end{align*}
The other two relations can be proved similarly with the help of
\begin{align}
&\varepsilon(\sigma)Y_{-\lambda_1}\cdots Y_{-\lambda_N}=Y_{-\lambda_{\sigma(1)}+\sigma(1)-1}\cdots Y_{-\lambda_{\sigma(i)}+\sigma(i)-i}\cdots Y_{-\lambda_{\sigma(N)}+\sigma(N)-N},\nonumber\\
&\varepsilon(\sigma)W_{-\lambda_1}\cdots W_{-\lambda_N}=W_{-\lambda_{\sigma(1)}+\sigma(1)-1}\cdots W_{-\lambda_{\sigma(i)}+\sigma(i)-i}\cdots W_{-\lambda_{\sigma(N)}+\sigma(N)-N},\label{e:sp1}
\end{align}
where $\sigma\in S_N$.
\end{proof}

By direct calculation, we have the following two lemmas.
\begin{Lemma}\label{le} From the commutation relations \eqref{e:he1} and \eqref{e:eq10}, we have
\begin{align*}
&\langle 0|\Gamma_+(\{x\})\Gamma_-(\{y\})|0\rangle = \prod^N_{i=1}\prod^K_{j=1}\big(1-x_iy_j\big)^{-1},
\end{align*}
where $x=(x_1,\dots,x_N),~y=(y_1,\dots,y_K)$.
\end{Lemma}
Now we study some cases of the inner product $\langle 0|\Gamma_+(\{x\})\Gamma_-(\{y\})|0\rangle$ related to symmetric functions. Due to \eqref{e:or}--\eqref{e:o3}, \eqref{e:sc8}, \eqref{e:sp12}, \eqref{e:o11} and Lemma~\ref{le5}, we have the following result.
\begin{Lemma}\label{le6} For $x=(x_1,\dots,x_N)$, $y=(y_1,\dots,y_N)$, $x^{\pm}=\big(x^{\pm}_1,\dots,x^{\pm}_N\big)$, we have
\begin{align*}
&\langle 0|\Gamma_+(\{x\})\Gamma_-(\{y\})|0\rangle=\sum_{\substack{\lambda\\l(\lambda)\leq N}}{\rm s}_\lambda(x){\rm s}_\lambda(y),\\ 
&\langle 0|\Gamma_+\big(\big\{x^{\pm}\big\}\big)\Gamma_-(\{y\})|0\rangle=\prod_{1\leq k< l\leq N}(1-y_ky_l)^{-1}\sum_{\substack{\lambda\\l(\lambda)\leq N}}{\rm sp}_\lambda\big(x^{\pm}\big){\rm s}_\lambda(y),\\ 
&\langle 0|\Gamma_+\big(\big\{x^{\pm}\big\}\big)\Gamma_-(\{y\})|0\rangle=\prod_{1\leq k\leq l\leq N}(1-y_ky_l)^{-1}\sum_{\substack{\lambda\\l(\lambda)\leq N}}{\rm o}_\lambda\big(x^{\pm}\big){\rm s}_\lambda(y). 
\end{align*}
\end{Lemma}

Combining Lemmas~\ref{le} and~\ref{le6}, we can prove the classical Cauchy identities associated to Schur functions, symplectic and orthogonal Schur functions respectively.
\begin{Theorem}[\cite{Be2018,KT1987,Su1990}]
Let $x=(x_1,\dots,x_N)$, $y=(y_1,\dots,y_N)$, $x^{\pm}=\big(x^{\pm}_1,\dots,x^{\pm}_N\big)$. Then
\begin{align*}
&\sum_{\substack{\lambda\\l(\lambda)\leq N}}{\rm s}_\lambda(x){\rm s}_\lambda(y)=\prod^{N}_{i,j=1}(1-x_iy_j)^{-1},\\
&\sum_{\substack{\lambda\\l(\lambda)\leq N}}{\rm sp}_\lambda\big(x^{\pm}\big){\rm s}_\lambda(y)=\prod^{N}_{i,j=1}\bigl[(1-x_iy_j)\big(1-x^{-1}_iy_j\big)\bigr]^{-1}\prod_{1\leq k< l\leq N}(1-y_ky_l),\\
&\sum_{\substack{\lambda\\l(\lambda)\leq N}}{\rm o}_\lambda\big(x^{\pm}\big){\rm s}_\lambda(y)=\prod^{N}_{i,j=1}\big[(1-x_iy_j)\big(1-x^{-1}_iy_j\big)\big]^{-1}\prod_{1\leq k\leq l\leq N}(1-y_ky_l),
\end{align*}
where each summation is over partitions $\lambda$ with $l(\lambda)\leq N$.
\end{Theorem}
The following is a generalization of Lemma~\ref{le5}.
\begin{Lemma}For a generalized partition $\mu=(\mu_1,\dots,\mu_l)$ and $x=(x_1,x_2,\dots,x_N)$, we have
\begin{align}
\label{e:sc19}&\Gamma_-(\{x\})|\mu\rangle=\sum_{\substack{\lambda\\l(\lambda)\leq l+N}}{\rm s}_{\lambda/\mu}(x)|\lambda\rangle,\\
\label{e:sp20}&\Gamma_-(\{x\})|\mu^{\rm sp}\rangle=\prod_{1\leq i<j\leq N}(1-x_ix_j)^{-1}\sum_{\lambda=(\lambda_1,\dots,\lambda_{l+N})}{\rm S}^*_{\lambda/\mu}(x)|\lambda^{\rm sp}\rangle,\\
&\Gamma_-(\{x\})|\mu^{\rm o}\rangle=\prod_{1\leq i\leq j\leq N}(1-x_ix_j)^{-1}\sum_{\lambda=(\lambda_1,\dots,\lambda_{l+N})}{\rm S}^*_{\lambda/\mu}(x)|\lambda^{\rm o}\rangle,\nonumber
\end{align}
where the second and third sums are over generalized partitions $\lambda=(\lambda_1,\dots,\lambda_{l+N})$ with zeros parts allowed in the end, and
\begin{align*}
{\rm S}^*_{\lambda/\mu}(x)=\frac{\det(c_{ij})^{l+N}_{i,j=1}}{\det\big(x^{N-j}_i\big)_{1\leq i,j\leq N}}
\end{align*}
with
\[
c_{ij}=\begin{cases}
x^{\lambda_j-j+N}_i, &1\leq i\leq N,\\
h_{\mu_{i-N}-\lambda_j+j-i}\big(x^{\pm}\big), & N+1\leq i\leq N+l.
\end{cases}
\]
\end{Lemma}
\begin{proof} We only prove \eqref{e:sp20}, as the others can be handled similarly. By definition \eqref{e:ve1}, we have
\begin{align*}
\Gamma_-(\{x\})=\frac{1}{\prod_{1\leq i<j\leq N}(1-x_ix_j)\big(1-x^{-1}_ix_j\big)}Y(x_1)\cdots Y(x_N)\Gamma_+\big(\big\{x^\pm\big\}\big),
\end{align*}
where $x^\pm=\big(x^\pm_1,\dots,x^\pm_N\big)$. It is easy to check that (by vertex operator calculus)
\begin{align*}
\Gamma_+\big(\big\{x^\pm\big\}\big)Y(z)=\frac{1}{\prod^N_{i=1}(1-x_iz)\big(1-x^{-1}_iz\big)}Y(z)\Gamma_+\big(\big\{x^\pm\big\}\big),
\end{align*}
thus
\begin{align*}
\Gamma_+\big(\big\{x^\pm\big\}\big)Y_n=\sum_{i\geq 0}h_i\big(x^\pm\big)Y_{n+i}\Gamma_+\big(\big\{x^\pm\big\}\big).
\end{align*}
Then
\begin{align*}
&\prod_{1\leq i<j\leq N}(1-x_ix_j)\big(1-x^{-1}_ix_j\big)\Gamma_-(\{x\})|\mu^{\rm sp}\rangle\\
&\qquad{}=Y(x_1)\cdots Y(x_N)\sum_{k_{N+1},\dots,k_{N+l}\geq 0}h_{k_{N+1}}\big(x^\pm\big)\cdots h_{k_{N+l}}\big(x^\pm\big)Y_{-\mu_1+k_{N+1}}\cdots Y_{-\mu_l+k_{N+l}}|0\rangle\\
&\qquad{}=\sum_{k_1,\dots,k_N\in \mathbb{Z}}x^{k_1}_1\cdots x^{k_N}_N\sum_{k_{N+1},\dots,k_{N+l}\geq 0}h_{k_{N+1}}\big(x^\pm\big)\cdots h_{k_{N+l}}\big(x^\pm\big)\\
&\qquad\quad{} \times Y_{-k_1}\cdots Y_{-k_N} Y_{-\mu_1+k_{N+1}}\cdots Y_{-\mu_l+k_{N+l}}|0\rangle.
\end{align*}
By \eqref{e:sp1}, the coefficient of $|\lambda^{\rm sp}\rangle$ in $\prod_{1\leq i<j\leq N}(1-x_ix_j)\big(1-x^{-1}_ix_j\big)\Gamma_-(\{x\})|\mu^{\rm sp}\rangle$ is
\begin{align*}
\det(c_{ij})^{l+N}_{i,j=1}
\end{align*}
with
\[
c_{ij}=\begin{cases}
x^{\lambda_j-j+i}_i, & 1\leq i\leq N,\\
h_{\mu_{i-N}-\lambda_j+j-i}\big(x^{\pm}\big), & N+1\leq i\leq N+l.
\end{cases}
\]
By means of the Vandermonde identity, we can get \eqref{e:sp20}.
\end{proof}

\begin{Remark}
The symmetric rational function ${\rm S}^*_{\lambda/\mu}(x)$ makes sense without the restriction $\mu\subset\lambda$, while ${\rm s}_{\lambda/\mu}(x)$ is zero unless $\mu\subset\lambda$. Unlike the skew symplectic Schur functions ${\rm sp}_{\lambda/\mu}\big(x^{\pm}\big)$, the symmetric
polynomials ${\rm S}^*_{\lambda/\mu}(x_1,\dots,x_N)$
do not change if we add zeroes to $\lambda$ and $\mu$ simultaneously.
More precisely,
\[
{\rm S}^*_{(\lambda_1,\dots,\lambda_{M+N-1},\lambda_{M+N})/(\mu_1,\dots,\mu_{M-1},0)}(x_1,\dots,x_N)=0,
\]
 unless $\lambda_{M+N}=0$, and
\[
{\rm S}^*_{(\lambda_1,\dots,\lambda_{M+N-1},0)/(\mu_1,\dots,\mu_{M-1},0)}(x_1,\dots,x_N)={\rm S}^*_{(\lambda_1,\dots,\lambda_{M+N-1})/(\mu_1,\dots,\mu_{M-1})}(x_1,\dots,x_N).
\]
Moreover,
${\rm S}^*_{(\lambda_1,\dots,\lambda_{M+N})/(0^M)}(x_1,\dots,x_N)=0$ unless $\lambda_{N+1}=\cdots=\lambda_{M+N}=0$,
 and
 \[
 {\rm S}^*_{(\lambda_1,\dots,\lambda_{N},0^M)/(0^M)}(x_1,\dots,x_N)={\rm S}^*_{(\lambda_1,\dots,\lambda_{N})/\varnothing}(x_1,\dots,x_N)
 ={\rm s}_{(\lambda_1,\dots,\lambda_{N})}(x_1,\dots,x_N).
 \]
\end{Remark}
We now offer a vertex algebraic proof of the Cauchy-type identity for skew Schur functions \mbox{\cite[p.~93]{Mac1995}}, which can be seen as a generalized Cauchy identity for Schur functions.
\begin{Theorem}For two partitions $\lambda$ and $\mu$,
\[
\sum_\rho {\rm s}_{\rho/\lambda}(x){\rm s}_{\rho/\mu}(y)=\prod_{i,j}(1-x_iy_j)^{-1}\sum_\nu {\rm s}_{\mu/\tau}(x){\rm s}_{\lambda/\tau}(y),
\]
where the sum is over all partitions $\rho$ and $\tau$ such that the skew Schur functions are defined.
\end{Theorem}
\begin{proof}
From \eqref{e:or}, we can easily deduce that
\begin{align*}
\langle\lambda|\sum_{\rho}|\rho\rangle\langle\rho|=\langle\lambda|,\qquad \sum_{\rho}|\rho\rangle\langle\rho||\lambda\rangle=|\lambda\rangle,
\end{align*}
where the sum is over all partitions $\rho$.
Thus
\begin{align}\label{e:sc17}
\langle \lambda|\Gamma_+(\{x\})\Gamma_-(\{y\})|\mu\rangle
=\langle \lambda|\Gamma_+(\{x\})\sum_{\rho}|\rho\rangle\langle\rho|\Gamma_-(\{y\})|\mu\rangle
=\sum_\rho {\rm s}_{\rho/\lambda}(x){\rm s}_{\rho/\mu}(y),
\end{align}
where we have used \eqref{e:sc6}. We also have the following relation:
\begin{align}
\notag\langle \lambda|\Gamma_+(\{x\})\Gamma_-(\{y\})|\mu\rangle &{}= \prod_{i,j}(1-x_iy_j)^{-1}\langle \lambda|\Gamma_-(\{y\})\Gamma_+(\{x\})|\mu\rangle\\
\notag &{}=\prod_{i,j}(1-x_iy_j)^{-1}\langle \lambda|\Gamma_-(\{y\})\sum_{\tau}|\tau\rangle\langle\tau|\Gamma_+(\{x\})|\mu\rangle\\
 &{}= \prod_{i,j}(1-x_iy_j)^{-1}\sum_\nu {\rm s}_{\mu/\tau}(x){\rm s}_{\lambda/\tau}(y).\label{e:sc18}
\end{align}
Comparing \eqref{e:sc17} and \eqref{e:sc18}, we get the identity.
\end{proof}

We also extend the vertex algebraic approach to derive the Cauchy-type identities for skew symplectic/orthogonal Schur functions.
\begin{Theorem}For generalized partition $\mu=(\mu_1,\dots,\mu_l)$ and $\lambda=(\lambda_1,\dots,\lambda_{l+N})$, one has
\begin{gather}
\sum_{\rho=(\rho_1,\dots,\rho_{l+N}) }{\rm sp}_{\rho/\mu}\big(x^{\pm}\big){\rm S}^*_{\rho/\lambda}(y){}\nonumber\\
\qquad\quad{}=\prod^N_{i=1}\prod^K_{j=1}\frac{1}{(1-x_iy_j)(1-x^{-1}_iy_j)}\sum_{\tau=(\tau_1,\dots,\tau_l)}{\rm sp}_{\lambda/\tau}\big(x^{\pm}\big){\rm S}^*_{\mu/\tau}(y),\label{e:sp16}\\
\sum_{\rho=(\rho_1,\dots,\rho_{l+N})}{\rm o}_{\rho/\mu}\big(x^{\pm}\big){\rm S}^*_{\rho/\lambda}(y)\nonumber\\
\qquad\quad{}=\prod^N_{i=1}\prod^K_{j=1}\frac{1}{(1-x_iy_j)(1-x^{-1}_iy_j)}\sum_{\tau=(\tau_1,\dots,\tau_l)}{\rm o}_{\lambda/\tau}\big(x^{\pm}\big){\rm S}^*_{\mu/\tau}(y),\label{e:o12}
\end{gather}
where $x^{\pm}=\big(x^{\pm}_1,\dots,x^{\pm}_N\big)$, $y=(y_1,\dots,y_K)$.
\end{Theorem}
\begin{proof}For generalized partition $\nu=(\nu_1,\dots,\nu_n)$, recall the relations \eqref{e:sp24} and \eqref{e:o24}
\begin{alignat*}{3}
&\langle\nu^{\rm sp}|\sum_{\eta=(\eta_1,\dots,\eta_n)}|\eta^{\rm sp}\rangle\langle\eta^{\rm sp}| = \langle\nu^{\rm sp}|,\qquad&& \sum_{\eta=(\eta_1,\dots,\eta_n)}|\eta^{\rm sp}\rangle\langle\eta^{\rm sp}||\nu^{\rm sp}\rangle = |\nu^{\rm sp}\rangle,&\\
&\langle\nu^{\rm o}|\sum_{\eta=(\eta_1,\dots,\eta_n)}|\eta^{\rm o}\rangle\langle\eta^{\rm o}| = \langle\nu^{\rm o}|,\qquad&& \sum_{\eta=(\eta_1,\dots,\eta_n)}|\eta^{\rm o}\rangle\langle\eta^{\rm o}||\nu^{\rm o}\rangle = |\nu^{\rm o}\rangle,&
\end{alignat*}
where the sum is over all generalized partitions $\eta=(\eta_1,\dots,\eta_n)$.
On the one hand,
\begin{align}\label{e:sp17}
\notag\langle\mu^{\rm sp}|\Gamma_+\big(\big\{x^{\pm}\big\}\big)\Gamma_-(\{y\})|\lambda^{\rm sp}\rangle=~&\langle\mu^{\rm sp}|\Gamma_+\big(\big\{x^{\pm}\big\}\big)\sum_{\rho=(\rho_1,\dots,\rho_{l+N}) }|\rho^{\rm sp}\rangle\langle\rho^{\rm sp}|\Gamma_-(\{y\})|\lambda^{\rm sp}\rangle\\
=~&\sum_{\rho=(\rho_1,\dots,\rho_{l+N})}{\rm sp}_{\rho/\mu}\big(x^{\pm}\big){\rm S}^*_{\rho/\lambda}(y),
\end{align}
where we used \eqref{e:sp15} and \eqref{e:sp20}. On the other hand,
\begin{align}
&\langle\mu^{\rm sp}|\Gamma_+\big(\big\{x^{\pm}\big\}\big)\Gamma_-(\{y\})|\lambda^{\rm sp}\rangle \nonumber\\
&\qquad\quad{}=\prod^N_{i=1}\prod^K_{j=1}\frac{1}{(1-x_iy_j)\big(1-x^{-1}_iy_j\big)}\langle\mu^{\rm sp}|\Gamma_-(\{y\})\sum_{\tau=(\tau_1,\dots,\tau_l) }|\tau^{\rm sp}\rangle\langle\tau^{\rm sp}|\Gamma_+\big(\big\{x^{\pm}\big\}\big)|\lambda^{\rm sp}\rangle \nonumber\\
&\qquad\quad{}=\prod^N_{i=1}\prod^K_{j=1}\frac{1}{(1-x_iy_j)\big(1-x^{-1}_iy_j\big)}\sum_{\tau=(\tau_1,\dots,\tau_l)}{\rm sp}_{\lambda/\tau}\big(x^{\pm}\big){\rm S}^*_{\mu/\tau}(y).\label{e:sp18}
\end{align}
Comparing \eqref{e:sp17} with \eqref{e:sp18}, we can get \eqref{e:sp16}. The relation \eqref{e:o12} can be proved similarly.
\end{proof}

\begin{Remark} The generalized partition $\tau$ in the right-hand side of \eqref{e:sp18} is contained in the generalized partition $\lambda$. For $\mu=\big(0^N\big)$, $\lambda=\big(0^N\big)$, relation \eqref{e:sp18} recovers the classical Cauchy identity for Schur functions
\begin{align*}
\sum_{\rho=(\rho_1,\dots,\rho_{N}) }{\rm s}_{\rho}\big(x_1,x^{-1}_1,\dots,x_N,x^{-1}_N\big){\rm s}_{\rho}(y_1,\dots,y_N)
=\frac{1}{\prod^N_{i=1}\prod^N_{j=1}(1-x_iy_j)\big(1-x^{-1}_iy_j\big)}
\end{align*}
due to Remark~\ref{e:remark1} and the fact that the generalized partition $\tau$ have to be $\big(0^N\big)$.
\end{Remark}
\begin{Remark}

In~\cite{Lam2006}, Lam studied some general symmetric functions via a boson-fermion correspondence. It will be interesting to study and realize the
skew symmetric functions using vertex operators.
\end{Remark}

{\bf Added in proof}. In the current paper posted on arXiv in 2022, among other things we have derived the Jacobi--Trudi formulas for the skew Schur symplectic functions and the skew Schur orthogonal functions in terms of the
(generalized) homogeneous symmetric functions. We recently noticed that~\cite{AFHS} has found a combinatorial proof of the alternative Jacobi--Trudi formulas in terms of elementary symmetric functions in 2023.
We will also give a vertex algebraic proof of the
latter in a forthcoming paper.

\appendix

\section{Another proof of Proposition~\ref{pro10}}\label{appA}
In this appendix, we give a simple vertex algebraic proof of Proposition~\ref{pro10} and \eqref{e:sc19}.
\begin{proof}
From the definition of $X(z)$, we have
\begin{align*}
\Gamma_+(\{x\})X(z)=\prod^N_{i=1}\frac{1}{1-x_iz}X(z)\Gamma_+(\{x\}),
\end{align*}
i.e.,
\begin{align*}
\Gamma_+(\{x\})X_n=\sum_{i\geq 0}h_i(x)X_{n+i}\Gamma_+(\{x\}).
\end{align*}
Thus
\begin{gather*}
\Gamma_+(\{x\})|\lambda\rangle
 =\Gamma_+(\{x\})X_{-\lambda_1}\dots X_{-\lambda_k}|0\rangle\\
 =\sum_{i_1,\dots,i_k\geq 0}h_{i_1}(x)\cdots h_{i_k}(x)X_{-\lambda_1+i_1}\dots X_{-\lambda_k+i_k}|0\rangle
\\
= \sum_{\nu}\!\sum_{\sigma\in S_k}\!\!\epsilon(\sigma)h_{\lambda_1-\nu_{\sigma(1)}+\sigma(1)-1}(x)\!\cdots h_{\lambda_k-\nu_{\sigma(k)}+\sigma(k)-k}(x)X_{-\nu_{\sigma(1)}+\sigma(1)-1}\!\cdots X_{-\nu_{\sigma(k)}+\sigma(k)-k}|0\rangle\\
= \sum_{\nu}\det(h_{\lambda_i-\nu_j-i+j}(x))^k_{i,j=1}|\nu\rangle = \sum_{\nu}{\rm s}_{\lambda/\nu}(x)|\nu\rangle,
\end{gather*}
where we have used \eqref{e:sc15} and the fact $\Gamma_+(\{x\})|0\rangle=|0\rangle$. By \eqref{e:or}, we have shown Proposition~\ref{pro10}. Similarly, using~\eqref{e:sc5} and~\eqref{e:or}, we can get
\begin{align*}
\langle\lambda|\Gamma_-(\{x\})|\mu\rangle = {\rm s}_{\lambda/\mu}(x).
\end{align*}
It is clear that
\begin{align*}
\Gamma_-(\{x\})|\mu\rangle = \sum_{\mu\subset \lambda}{\rm s}_{\lambda/\mu}(x)|\lambda\rangle.\tag*{\qed}
\end{align*}\renewcommand{\qed}{}
\end{proof}

\subsection*{Acknowledgements}

We extend our heartfelt appreciation to the anonymous referees for their invaluable constructive feedback and suggestions, which have enhanced the quality of the paper.
The research is supported by the Simons Foundation (grant no.\ MP-TSM-00002518), NSFC (grant nos.\ 12171303, 12101231, 12301033), and NSF of Huzhou (grant no.\ 2022YZ47).

\pdfbookmark[1]{References}{ref}
\LastPageEnding


\begin{thebibliography}{99}
\footnotesize\itemsep=0pt

\bibitem{AFHS}
Albion S.P., Fischer I., H\"ongesberg H., Schreier-Aigner F., Skew symplectic
 and orthogonal characters through lattice paths, \href{https://arxiv.org/abs/2305.11730}{arXiv:2305.11730}.

\bibitem{AF2020}
Ayyer A., Fischer I., Bijective proofs of skew {S}chur polynomial
 factorizations, \href{https://doi.org/10.1016/j.jcta.2020.105241}{\textit{J.~Combin. Theory Ser.~A}} \textbf{174} (2020),
 105241, 40~pages, \href{https://arxiv.org/abs/1905.05226}{arXiv:1905.05226}.

\bibitem{Ba1996}
Baker T.H., Vertex operator realization of symplectic and orthogonal
 {$S$}-functions, \href{https://doi.org/10.1088/0305-4470/29/12/017}{\textit{J.~Phys.~A}} \textbf{29} (1996), 3099--3117.

\bibitem{Be2018}
Betea D., Correlations for symplectic and orthogonal {S}chur measures,
 \href{https://arxiv.org/abs/1804.08495}{arXiv:1804.08495}.

\bibitem{CG2020}
Cuenca C., Gorin V., {$q$}-deformed character theory for infinite-dimensional
 symplectic and orthogonal groups, \href{https://doi.org/10.1007/s00029-020-00572-8}{\textit{Selecta Math.~(N.S.)}} \textbf{26}
 (2020), 40, 55~pages, \href{https://arxiv.org/abs/1812.06523}{arXiv:1812.06523}.

\bibitem{DJKM}
Date E., Kashiwara M., Jimbo M., Miwa T., Transformation groups for soliton
 equations, in Nonlinear {I}ntegrable {S}ystems~-- {C}lassical {T}heory and
 {Q}uantum {T}heory ({K}yoto, 1981), World Scientific Publishing, Singapore,
 1983, 39--119.

\bibitem{FJK2010}
Fauser B., Jarvis P.D., King R.C., Plethysms, replicated {S}chur functions and
 series, with applications to vertex operators, \href{https://doi.org/10.1088/1751-8113/43/40/405202}{\textit{J.~Phys.~A}}
 \textbf{43} (2010), 405202, 30~pages, \href{https://arxiv.org/abs/1006.4266}{arXiv:1006.4266}.

\bibitem{FK1980}
Frenkel I.B., Kac V.G., Basic representations of affine {L}ie algebras and dual
 resonance models, \href{https://doi.org/10.1007/BF01391662}{\textit{Invent. Math.}} \textbf{62} (1980), 23--66.

\bibitem{Jing1991}
Jing N., Vertex operators, symmetric functions, and the spin group
 {$\Gamma_n$}, \href{https://doi.org/10.1016/0021-8693(91)90177-A}{\textit{J.~Algebra}} \textbf{138} (1991), 340--398.

\bibitem{Jing2000}
Jing N., Symmetric polynomials and {$U_q(\widehat{\rm sl}{}_2)$},
 \href{https://doi.org/10.1090/S1088-4165-00-00065-0}{\textit{Represent. Theory}} \textbf{4} (2000), 46--63,
 \href{https://arxiv.org/abs/math.QA/9902109}{arXiv:math.QA/9902109}.

\bibitem{JL2020}
Jing N., Li Z., A~note on {C}auchy's formula, \href{https://doi.org/10.1016/j.aam.2023.102630}{\textit{Adv. in Appl. Math.}}
 \textbf{153} (2024), 102630, 14~pages, \href{https://arxiv.org/abs/2202.11175}{arXiv:2202.11175}.

\bibitem{JN2015}
Jing N., Nie B., Vertex operators, {W}eyl determinant formulae and {L}ittlewood
 duality, \href{https://doi.org/10.1007/s00026-015-0271-z}{\textit{Ann. Comb.}} \textbf{19} (2015), 427--442,
 \href{https://arxiv.org/abs/1308.2792}{arXiv:1308.2792}.

\bibitem{JR2016}
Jing N., Rozhkovskaya N., Vertex operators arising from {J}acobi--{T}rudi
 identities, \href{https://doi.org/10.1007/s00220-015-2564-9}{\textit{Comm. Math. Phys.}} \textbf{346} (2016), 679--701,
 \href{https://arxiv.org/abs/1411.4725}{arXiv:1411.4725}.

\bibitem{KT1987}
Koike K., Terada I., Young-diagrammatic methods for the representation theory
 of the classical groups of type {$B_n$}, {$C_n$}, {$D_n$},
 \href{https://doi.org/10.1016/0021-8693(87)90099-8}{\textit{J.~Algebra}} \textbf{107} (1987), 466--511.

\bibitem{KT1990}
Koike K., Terada I., Young diagrammatic methods for the restriction of
 representations of complex classical {L}ie groups to reductive subgroups of
 maximal rank, \href{https://doi.org/10.1016/0001-8708(90)90059-V}{\textit{Adv. Math.}} \textbf{79} (1990), 104--135.

\bibitem{Lam2006}
Lam T., A combinatorial generalization of the {B}oson-{F}ermion correspondence,
 \href{https://doi.org/10.4310/MRL.2006.v13.n3.a4}{\textit{Math. Res. Lett.}} \textbf{13} (2006), 377--392,
 \href{https://arxiv.org/abs/math.CO/0507341}{arXiv:math.CO/0507341}.

\bibitem{LTW2016}
Li H., Tan S., Wang Q., A certain {C}lifford-like algebra and quantum vertex
 algebras, \href{https://doi.org/10.1007/s11856-016-1416-4}{\textit{Israel~J. Math.}} \textbf{216} (2016), 441--470,
 \href{https://arxiv.org/abs/1505.07165}{arXiv:1505.07165}.

\bibitem{Lw1950}
Littlewood D.E., The theory of group characters and matrix representations of
 groups, Oxford University Press, New York, 1940.

\bibitem{Mac1995}
Macdonald I.G., Symmetric functions and {H}all polynomials, 2nd~ed., \textit{Oxford
 Math. Monogr.}, The Clarendon Press, \href{https://doi.org/10.1093/oso/9780198534891.001.0001}{Oxford University Press}, New York, 1995.

\bibitem{Me2019}
Meckes E.S., The random matrix theory of the classical compact groups,
 \textit{Cambridge Tracts in Math.}, Vol.~218, \href{https://doi.org/10.1017/9781108303453.009}{Cambridge University Press},
 Cambridge, 2019.

\bibitem{Oko2001}
Okounkov A., Infinite wedge and random partitions, \href{https://doi.org/10.1007/PL00001398}{\textit{Selecta
 Math.~(N.S.)}} \textbf{7} (2001), 57--81, \href{https://arxiv.org/abs/math.RT/9907127}{arXiv:math.RT/9907127}.

\bibitem{Pro1994}
Proctor R.A., Young tableaux, {G}el'fand patterns, and branching rules for
 classical groups, \href{https://doi.org/10.1006/jabr.1994.1064}{\textit{J.~Algebra}} \textbf{164} (1994), 299--360.

\bibitem{SZ2006}
Shimozono M., Zabrocki M., Deformed universal characters for classical and
 affine algebras, \href{https://doi.org/10.1016/j.jalgebra.2006.03.008}{\textit{J.~Algebra}} \textbf{299} (2006), 33--61,
 \href{https://arxiv.org/abs/math.CO/0404288}{arXiv:math.CO/0404288}.

\bibitem{Sta1999}
Stanley R.P., Enumerative combinatorics. {V}ol.~2, \textit{Cambridge Stud. Adv.
 Math.}, Vol.~62, \href{https://doi.org/10.1017/CBO9780511609589}{Cambridge University Press}, Cambridge, 1999.

\bibitem{Su1990}
Sundaram S., Tableaux in the representation theory of the classical {L}ie
 groups, in Invariant Theory and Tableaux ({M}inneapolis, {MN}, 1988),
 \textit{IMA Vol. Math. Appl.}, Vol.~19, Springer, New York, 1990, 191--225.

\bibitem{Wey1946}
Weyl H., The classical groups. {T}heir invariants and representations, \textit{Princeton Math. Ser.}, \href{https://doi.org/10.1515/9781400883905}{Princeton University Press}, Princeton, NJ, 1946.

\end{thebibliography}
\end{document}